\documentclass[a4paper,11pt]{article}
\usepackage[margin=3cm]{geometry}

\usepackage[T1]{fontenc}
\usepackage[utf8]{inputenc}

\usepackage{amssymb,amsmath,amsthm}
\usepackage{enumerate}
\usepackage{hyperref}  

\usepackage{dsfont}  	
\usepackage{mathtools}  
\usepackage{tikz-cd}    

\usepackage[draft,inline,nomargin,index]{fixme} 
\fxsetup{theme=color,mode=multiuser}			
\FXRegisterAuthor{fh}{afh}{\color{red}FH}		
\FXRegisterAuthor{flm}{aflm}{\color{blue}FLM}	

\title{Spatial models for boolean actions in the infinite measure-preserving setup}
\author{Fabien Hoareau and François Le Maître}

\newcommand{\abs}[1]{\left\lvert #1\right\rvert}
\newcommand{\norm}[1]{\left\lVert #1\right\rVert}
\newcommand{\inv}{^{-1}}

\renewcommand{\leq}{\leqslant}
\renewcommand{\geq}{\geqslant}

\newcommand{\C}{\mathbb{C}}
\newcommand{\Z}{\mathbb{Z}}
\newcommand{\Q}{\mathbb{Q}}

\newcommand{\R}{\mathbb{R}}
\newcommand{\N}{\mathbb{N}}

\newcommand{\LL}{\mathrm{L}}
\newcommand{\PP}{\mathbb{P}}
\newcommand{\Sinf}{\mathfrak{S}_{\infty}}
\newcommand{\LON}{\mathrm{LO}(\N)}

\newcommand{\Aut}{\mathrm{Aut}}
\newcommand{\MAlg}{\mathrm{MAlg}}
\newcommand{\Leb}{\mathrm{Leb}}

\newcommand{\id}{\mathrm{id}}
\DeclareMathOperator{\supp}{\mathrm{supp}}
\DeclareMathOperator{\spec}{\mathrm{Sp}}

\newcommand{\Poisson}[1]{\mathrm{Pois}^{#1}}
\newcommand{\PPP}[2]{\mathbb P^{#1}_{#2}}

\newtheorem{thm}{Theorem}[section]
\newtheorem{cor}[thm]{Corollary}
\newtheorem{lem}[thm]{Lemma}
\newtheorem{prop}[thm]{Proposition}

\newtheorem{thmi}{Theorem}				

\theoremstyle{definition}

\newtheorem{defi}[thm]{Definition}

\newtheorem{rem}[thm]{Remark}

\newtheorem{question}{Question}
\newtheorem*{ack}{Acknowledgements}

\begin{document}

\maketitle

\begin{abstract}
   We show that up to a null set, every infinite measure-preserving action of a locally compact Polish group
   can be turned into a continuous measure-preserving action on a locally compact Polish space where the 
   underlying measure is Radon.
   We also investigate the distinction between spatial and boolean actions in the infinite measure-preserving setup. 
   In particular, we extend Kwiatkowska and Solecki's Point Realization Theorem
   to the infinite measure setup. 
   We finally obtain a streamlined proof of a recent result of Avraham-Re'em and Roy: Lévy groups cannot admit nontrivial continuous measure-preserving actions on Polish spaces when the measure is locally finite.
\end{abstract}

	{
		\small	
		\noindent\textbf{{MSC-classification:}}	
		 	\textbf{Primary}: 22F10; 37A15; 37A40. \textbf{Secondary}: 22D12; 60G55.
	}
\tableofcontents

\section{Introduction}

The notion of flow is a convenient framework for understanding global solutions of differential equations. In full generality, a \textbf{flow} on a set $X$ is an action
$\alpha:\R\times X\to X$ of the real line $\R$ on $X$. The connection with differential equations (on say $\R^n$ for simplicity) is simply the following: a \emph{global solution} of the differential equation associated to the vector field $V:\R^n\to\R^n$ is a flow $\alpha:\R\times \R^n\to \R^n$ such that for all $x\in\R^n$,
\[
\frac d{dt} \left[\alpha(t,x)\right]_{\restriction t=0} =V(x).
\] 
In particular, for all $x\in\R^n$, the map $t\mapsto \alpha(t,x)$ is everywhere differentiable. 
Provided a global solution exists and is unique, the regularity of the corresponding flow usually reflects that of the vector field. 
In particular, if $V$ is continuous, we cannot expect the flow to be something better than continuous in general.

Many flows arising as solutions of a differential equation preserve a natural $\sigma$-finite measure which comes from the geometry of the ambient manifold, thus allowing one to use ergodic-theoretic methods. 
A notable example is given by Hopf's study of the geodesic flow on hyperbolic surfaces \cite{hopfErgodicTheoryGeodesic1971}. From the ergodic point of view, flows are identified up to \textbf{conjugacy}, and make sense even when the flow is only assumed to be a measurable map. 

To be more precise, recall that a \textbf{standard Borel space} is an uncountable measurable space $X$ endowed with a $\sigma$-algebra $\mathcal B(X)$ of Borel subsets coming from some Polish topology on $X$.
All such spaces are isomorphic, see \cite[Thm.~15.6]{kechrisClassicalDescriptiveSet1995}, 
justifying the terminology. 
A \textbf{Borel flow} on $X$ is then an action $\alpha:\R\times X\to X$ 
which is a Borel map $\R\times X\to X$, meaning that it is measurable if we endow 
$\R\times X$ with the usual product $\sigma$-algebra of the Borel 
$\sigma$-algebra of $\R$ and that of $X$. 
Given a Borel $\sigma$-finite measure $\lambda$ on $X$, 
a \textbf{measure-preserving flow} on $(X,\lambda)$ is simply a Borel flow 
such that for all $t\in\R$, the (automatically Borel) bijection $\alpha(t,\cdot)$ 
satisfies $\lambda(A)=\lambda(\alpha(t,A))$ for all $A\subseteq X$ Borel. 
Finally, two flows $\alpha,\beta$ on respective $\sigma$-finite measured standard Borel spaces 
$(X,\lambda)$ and $(Y,\eta)$ are called  \textbf{spatially isomorphic} 
when there is a measure-preserving Borel bijection $\Phi:X_0\to Y_0$ between two 
invariant conull Borel subsets $X_0\subseteq X$ and $Y_0\subseteq Y$ 
such that for all $x\in X_0$ and all $t\in\R$, we have 
\[
\Phi(\alpha(t,x))=\beta(t,\Phi(x)).
\]
Given that certain measure-preserving flows come from differential equations, 
it is natural to wonder whether conversely every measure-preserving flows 
is spatially isomorphic to a flow coming from a differential equation. 
Our main result shows that as far as the \emph{topology}  is concerned, 
every measure-preserving flow is spatially isomorphic to a  flow 
where the topology and the measure play nicely with each other, 
just as they do when they come from a differential equation. 
The key point is that the flow is continuous while the measure is \emph{locally finite} 
(which implies that it is Radon since the ambient space is locally compact Polish, 
see Section \ref{sec: locally finite and Radon}). 
We state our result below in its most general form, 
replacing $\R$ by an arbitrary locally compact second-countable group, see Section \ref{sec: prelim action} for relevant definitions.

\begin{thmi}[{see Cor.~\ref{cor: radon model for lcsc via Mackey}}]{\label{thm intro: spatial conjugacy to nice action}}
	Let $(X,\lambda)$ be a standard $\sigma$-finite space, 
	let $G$ be a locally compact Polish group 
	and let $\alpha: G \times X \to X$ be a measure-preserving $G$-action on $X$. 
	Then $\alpha$ is spatially isomorphic to a continuous measure-preserving $G$-action 
	on a locally compact Polish space $Y$ endowed with a Radon  measure $\eta$. 
\end{thmi}

This result is new only when the measure $\lambda$ is infinite. In order to explain why, we first recall the Becker-Kechris theorem, which holds in the even wider setup of Polish groups actions and does not require one to throw out null sets. Note that for locally compact Polish groups, item \eqref{item:universalcompactGspace} below was first proven by Varadarajan, see \cite[Thm.~3.2]{varadarajanGroupsAutomorphismsBorel1963}.

\begin{thm}[{\cite[Thm.~2.2.6 and Thm.~5.2.1]{beckerDescriptiveSetTheory1996}}]\label{thm: bk spatial model} 
Let $G$ be a Polish group and let $\alpha: G\times X\to X$ be a Borel $G$-action on a standard Borel space $X$. 
    \begin{enumerate}[(1)]
    \item  \label{item:universalcompactGspace} There is a compact Polish space $K$ endowed with a continuous $G$-action $\beta:G\times K\to K$, and a Borel injection $\Phi:X\to K$ which is $G$-equivariant: for all $x\in X$ and all $g\in G$, 
    \[
    \Phi(\alpha(g,x))=\beta(g,\Phi(x)).
    \]
    Moreover, $\beta$ and $K$ can be chosen so as not to depend on $\alpha$.
    \item  \label{item:everyactionhasaspatialmodel} There exists a Polish topology on $X$ inducing its Borel structure such that $\alpha$ is continuous. 
    \end{enumerate}
\end{thm}

From the first item, we deduce that if $\lambda$ is any Borel finite measure on $X$ and $\alpha: G\curvearrowright (X,\lambda)$ is  any measure-preserving  action of a Polish group $G$, then the map $\Phi: X\to K$ is a conjugacy of $\alpha$ with the continuous measure-preserving action $\beta:G\curvearrowright (K,\Phi_*\lambda)$, where $K$ is endowed with the pushforward measure $\Phi_*\lambda$. Since finite Borel measures on compact Polish spaces are automatically Radon, \emph{when $\lambda$ is finite, Theorem \ref{thm intro:  spatial conjugacy to nice action} holds for any measure-preserving action of any Polish group $G$.} 
So our Theorem \ref{thm intro: spatial conjugacy to nice action} 
says something new only for infinite measures 
(in the above argument the pushforward measure $\Phi_*\lambda$ fails to be locally finite, 
being an infinite measure over a compact space). 
Actually, we have two ways of proving this result. 
The first uses the associated \emph{boolean action} (more on that later). 
The second one is the following measured version of the first item of the Becker-Kechris 
theorem.

\begin{thmi}[{see Thm.~\ref{thm: lc polish embed into radon}}]
\label{thmi: locally compact spatical into continuous Radon}   
    Let $(X,\lambda)$ be a standard $\sigma$-finite space, 
	let $G$ be a locally compact Polish group 
	and let $\alpha: G \times X \to X$ be a measure-preserving $G$-action on $X$. 
    Then there is a  continuous measure-preserving $G$-action 
	on a locally compact Polish space $Y$ and a Borel $G$-equivariant injection $\Phi: X\to Y$ such that the pushforward measure $\Phi_*\lambda$ is a Radon measure on $Y$. 
\end{thmi}

The above result is a strengthening of Theorem 
\ref{thm intro: spatial conjugacy to nice action}: 
instead of just a full measure subset of $X$, 
the whole set $X$ is taken into a continuous action with Radon measure.
The main idea of its proof is to adapt Glasner, Tsirelson and Weiss' 
notion of $G$-continuous function (see \cite[Def.~2.1]{glasnerAutomorphismGroupGaussian2005}) to the Borel setup, and to show that there is a countable set of integrable $G$-continuous bounded Borel functions that separates points. 
In order to prove the latter result, we use convolution, which is 
also crucial in our proof of Theorem \ref{thm intro: spatial conjugacy to nice action}
via boolean actions. 

Given Theorem \ref{thm intro: spatial conjugacy to nice action}, the following question is very natural.

\begin{question}
Let $G$ be a non locally compact Polish group.
   Let $\alpha:G\curvearrowright (X,\lambda)$ be a measure-preserving action on an infinite $\sigma$-finite standard measured space $(X,\lambda)$, is $\alpha$ spatially isomorphic to a measure-preserving continuous $G$-action $\beta$ on $(Y,\eta)$ with $Y$ locally compact Polish and $\eta$ Radon?
\end{question}

One can be more conservative and only require $Y$ to be Polish and $\eta$ to be locally finite.
There is however one obstruction to a positive answer: it may happen that for every $x\in X$ and every neighborhood of the identity $V\subseteq G$, the measurable set $\alpha(V,x)$ always has infinite measure (see Section \ref{sec: non radonable actions} for an example when $G=\mathfrak S_\infty$ is the permutation group of the integers). Other than that, we do not see any obstruction and the question is wide open.

Motivated by item \eqref{item:everyactionhasaspatialmodel} 
from Theorem \ref{thm: bk spatial model}, 
we ask another question.
\begin{question}\label{qu: loc fin Polish real}
    Let $\alpha$ be a measure-preserving action of a Polish group $G$ on an infinite $\sigma$-finite standard measured space $(X,\lambda)$.
    When can one endow $X$ itself with a Polish topology inducing its standard Borel structure so that the action $\alpha$ becomes continuous and the measure $\lambda$ becomes locally finite?
\end{question}

In the first version of this paper, we mentioned that we did not know the answer to this question even in the case 
where the acting group $G$ is locally compact Polish. However, Nachi Avraham Re'em kindly pointed out to us that 
Theorem \ref{thmi: locally compact spatical into continuous Radon} can be combined with a stronger version of the 
Becker-Kechris theorem so as to obtain a positive answer when $G$ is locally compact. 
We discuss this in Section \ref{sec: loc compact Radon vs loca fin real}, where we also 
explain why, even when $G$ is discrete, one cannot in general endow $X$ with a \emph{locally compact} Polish topology 
so that the $G$-action becomes continuous (see Proposition \ref{prop: no loc compact real}).\\

We now move to the topic of boolean measure-preserving actions. These can be defined in many equivalent ways (see Section \ref{sec: prelim}), but for this introduction the easiest is probably to describe them as continuous group homomorphisms $G\to \Aut(X,\lambda)$, where $\Aut(X,\lambda)$ is the group of all measure-preserving bijections of $(X,\lambda)$ \emph{identified up to measure zero}. 
A fundamental question is: which Boolean actions lift to genuine spatial actions? 

In order to answer it, Glasner, Tsirelson and Weiss developped in the setup of a probability measure-preserving boolean action $\alpha$ the following notion which we already briefly mentioned: a function $f\in\LL^\infty(X,\lambda)$ is \textbf{$G$-continuous} if $\norm{f-f\circ \alpha(g,\cdot)}_\infty\to 0$ as $g\to e_G$. 
Here is a natural extension of the conclusion of \cite[Thm.\ 1.7]{glasnerSpatialNonspatialActions2005} to the case where $\lambda$ is possibly infinite (but still $\sigma$-finite). 

\begin{thmi}[see Thm.\ \ref{thm: chara existence spatial model}]
\label{thmi: chara existence spatial model}
Let $G$ be a Polish group, and let $\alpha:G\to \Aut(X,\lambda)$ be a boolean measure-preserving $G$-action on a standard $\sigma$-finite space $(X,\lambda)$. The following are equivalent:
\begin{enumerate}[(1)]
\item \label{itemi: dense G continuous L1}the algebra $\mathcal{G}$ of $G$-continuous functions satisfies  
\[
\overline{\mathcal G\cap \LL^1(X,\lambda)}^{\norm\cdot_1}=\LL^1(X,\lambda) ;
\]
\item \label{itemi: dense G continuous L2}the algebra $\mathcal{G}$ of $G$-continuous functions satisfies  
\[
\overline{\mathcal G\cap \LL^2(X,\lambda)}^{\norm\cdot_2}=\LL^2(X,\lambda) ;
\]
\item \label{itemi: admit continuous model}
one can endow $X$ with a locally compact Polish topology inducing its standard Borel space structure so that $\lambda$ is Radon and $\alpha$ lifts to a continuous measure-preserving $G$-action on $X$.
\end{enumerate}
\end{thmi}

\begin{rem}
    We don't know if condition \eqref{itemi: admit continuous model} can be changed so as to only require that $X$ is Polish and $\lambda$ is locally finite.
\end{rem}

We use the above theorem to obtain natural infinite measure-preserving version of results 
of Kwiatkowska-Solecki, namely \cite[Thm.~1.1]{kwiatkowskaSpatialModelsBoolean2011}. 
Recall that given a metric space $(X,d)$, its isometry group is naturally endowed
with the topology of pointwise convergence, which makes
it a Polish group as soon as $(X,d)$ is separable and complete.
While every Polish group arises as the isometry group of a 
separable complete metric space, the class of isometry groups 
of locally compact separable metric spaces is smaller, 
but it contains all locally compact Polish groups, and all Polish non-archimedean groups
(see \cite{kwiatkowskaSpatialModelsBoolean2011} for details). 
We can now state our generalization of 
\cite[Thm.~1.1]{kwiatkowskaSpatialModelsBoolean2011}, which itself extends
\cite[Thm.~1.4(a) and Thm.~2.3]{glasnerSpatialNonspatialActions2005}.

\begin{thmi}[{see Thm.~\ref{thm: iso lc spatial model}}]
\label{thmi: iso lc spatial model}
Let $G$ be the isometry group of a locally compact separable metric space.
Then every boolean measure-preserving $G$-action on a standard $\sigma$-finite space 
$(X,\lambda)$ can be lifted to a continuous measure-preserving action on $X$ where 
$X$ is endowed with a locally compact Polish topology inducing its 
standard Borel structure and $\lambda$ is a Radon measure.
\end{thmi}

When $G$ is locally compact, this strengthens a result of Mackey which provides a purely Borel realization (seee Theorem \ref{thm: mackey} and \cite{mackeyPointRealizationsTransformation1962} for the more general non-singular version). 
Furthermore, if the measure is in addition finite, our theorem follows by combining Mackey's result with \cite[Thm.~3.2]{varadarajanGroupsAutomorphismsBorel1963}, but our proof strategy is very different.
Indeed, we first treat the locally compact case by  
relying on convolution, 
which is very natural in the framework of $G$-continuous functions (see Theorem \ref{thm: continuous radon for lcsc}). 
As we briefly mentioned earlier, Theorem \ref{thmi: iso lc spatial model}
can be used to give another proof 
of Theorem \ref{thm intro: spatial conjugacy to nice action} 
using a result of Mackey (see Corollary \ref{cor: radon model for lcsc via Mackey}).

In the general case, our proof is similar to Kwiatkowska and Solecki's, except that 
we begin with a simpler fixed point argument in $\LL^2$ relying on closed convex hulls, 
and then 
we work directly at the level 
of subalgebras of $\LL^\infty(X,\lambda)$ 
(see the proof of Proposition~\ref{prop: dense Gcont for iso of lc}).

We also adapt the Glasner-Tsirelson-Weiss results on whirly actions (see \cite[Sec.~3]{glasnerAutomorphismGroupGaussian2005}) to the infinite measure-presering setup,
showing in Section \ref{sec: whirly} that the tautological boolean $\Aut(X,\lambda)$-action cannot be lifted to a Borel spatial action. \\

Finally, we obtain the following result on Lévy groups, analogous to the Glasner-Tsirelson-Weiss result that says that all their continuous probability measure-preserving actions on compact metric spaces are trivial \cite[Thm.~1.1]{glasnerAutomorphismGroupGaussian2005}.

\begin{thmi}[see Thm.~\ref{thm: Lévy infinite}]
\label{thmi: Lévy infinite}
    Every continuous measure-preserving spatial action of a Lévy group $G$ on a Polish space $X$ endowed with a $\sigma$-finite atomless locally finite measure $\lambda$ is trivial, \textit{i.e.}\ the set of fixed points $\{ x \in X \mid \forall g \in G , g \cdot x =  x \}$ is conull.
\end{thmi}

While finishing the first version of our paper, we were informed by Emmanuel Roy 
that Nachi Avraham-Re'em and him had obtained the exact same result as above and 
put it on arXiv quite recently 
(see \cite[Thm.~5]{avraham-reemPoissonianActionsPolish2024}). 
Our proof actually follows the same two steps as theirs.  
First, we construct a Poisson point process without using counting measures, 
by directly working in the Effros standard Borel space (see Theorem \ref{thm: PPP}, 
which is a minimalistic version of Theorem 3 in their paper). 
And then we conclude the proof using the Glasner-Tsirelson-Weiss theorem: 
the natural action of the group on the space of closed subsets preserves the 
Poisson point process probability measure, so it has to be trivial, 
which implies the triviality of the action we started with.

Since our construction of the Poisson point process appears to be more direct
than Avraham-Re'em and Roy's, we decided to keep it in our paper, 
but the reader should definitely consult their work: 
they obtain a more precise description of the Poisson point process 
along with very nice applications of the above result to diffeomorphism groups, 
using the Maharam extension. 
They also have a complete description of the boolean probability measure-preserving actions 
one can obtain through this Poisson point process construction 
when starting from an infinite measure-preserving boolean action.

Finally, we don't know whether the conclusion of Theorem \ref{thmi: Lévy infinite} 
holds for general infinite measure-preserving Borel actions, 
but recall that in general there are measure-preserving actions     
which cannot be spatially isomorphic to actions satisfying the hypothesis 
of the above theorem (see Section \ref{sec: non radonable actions}). 
Some related and intriguing open questions may be found in the last section 
of \cite{avraham-reemPoissonianActionsPolish2024}.

\begin{ack}
    We would like to thank Matthieu Joseph and Sam Mellick for enlightening conversations around these topics.
    We are also grateful to Todor Tsankov for allowing us to include an example of his (Proposition \ref{prop: todor example}), 
    to Georges Skandalis for pointing out the crucial Proposition \ref{prop:FaitGeorges}
    and to Nachi Avraham-Re'em for allowing us to include Corollary \ref{cor: loc fin model for lc groups}
    in the second version of our paper. 
    We are thankful to the first referee for their numerous remarks
    and suggestions which helped 
    improve the paper in many ways, 
    and to the second referee for their comments and 
    for asking us
    whether Theorem~\ref{thmi: iso lc spatial model} could be extended to the class 
    of isometry groups of locally compact separable metric spaces.
    Finally, we thank Emmanuel Roy for reaching out to us about our shared mathematical interests.  
\end{ack}

\section{Preliminaries}\label{sec: prelim}

\subsection{Spatial and boolean actions on standard \texorpdfstring{$\sigma$}{sigma}-finite spaces} \label{sec: prelim action}

Recall that a \textbf{standard Borel space} is an uncountable measurable space $X$ endowed with a $\sigma$-algebra $\mathcal B(X)$ of Borel subsets coming from some Polish topology on $X$.
All such spaces are isomorphic, see \cite[Thm.~15.6]{kechrisClassicalDescriptiveSet1995}, justifying the terminology.

Throughout the paper, $(X,\lambda)$ will denote a standard Borel  space $X$ endowed with a nontrivial atomless $\sigma$-finite measure $\lambda$.
Since rescaling the measure does not change the dynamics that we are interested in, there are only two fundamentally different cases:
\begin{itemize}
	\item $\lambda(X)<+\infty$, in which case $(X,\lambda)$ is isomorphic to $([0,\lambda(X)), \Leb_{\restriction})$ where $\Leb$ is the Lebesgue measure (this is a direct consequence of \cite[Thm.~17.41]{kechrisClassicalDescriptiveSet1995}); 
	\item $\lambda(X)=+\infty$, in which case $(X,\lambda)$ is isomorphic to $(\R,\Leb)$ (to see this, observe that by $\sigma$-finiteness, $X$ can be cut into a countable partition $(X_n)$ such that $\lambda(X_n)<+\infty$ for all $n$, then cut $\R$ into a countable partition consisting of  half-open intervals $I_n$ of length $\Leb(X_n)$ and apply the previous case).
\end{itemize}

Given a standard $\sigma$-finite space $(X,\lambda)$, a \textbf{measure-preserving bijection} of $(X,\lambda)$ is a Borel bijection $T:X\to X$ such that $T_*\lambda=\lambda$, that is, for all $A\subseteq X$ Borel, $\lambda(A)=\lambda(T\inv(A))$. 

\begin{defi}
	\label{def: spatial mp action}
	Let $G$ be a Polish group, let $(X,\lambda)$ be a standard $\sigma$-finite space. A \textbf{spatial} measure-preserving  $G$-action on $(X,\lambda)$ is a Borel action map 
	$\alpha:G\times X\to X$ such that for all $g\in G$, the map $\alpha(g,\cdot):X\to X$ defines a measure-preserving bijection of $(X,\lambda)$. 
\end{defi}

Our main goal in this paper is to contrast this notion to that of a boolean action, which to our knowledge first appeared in a paper of Mackey \cite{mackeyPointRealizationsTransformation1962}. We will give in the next section several characterizations of this notion, one of which justifies the terminology and the connection with Mackey's definition. In order to motivate the definition we are giving here, notice that by definition the fact that $\alpha$ is an action map means that for all $g,h\in G$ and for \emph{all} $x\in X$ and, we have 
$$\alpha(gh,x)=\alpha(g, \alpha(h,x)).$$

\begin{defi}
	\label{def: boolean mp or ns action}
	Let $G$ be a Polish group, let $(X,\lambda)$ be a standard $\sigma$-finite space. 
	A \textbf{boolean} measure-preserving $G$-action on $(X,\lambda)$ is a Borel map 
	$\alpha:G\times X\to X$ such that: 
	\begin{enumerate}[(i)]
		\item
		\label{cond: main condition for boolean action}
		for all $g,h\in G$, we have that for $\lambda$-\emph{almost} all $x\in X$, 
		\[
		\alpha(gh,x)=\alpha(g,\alpha(h,x));
		\]
		\item 
		\label{cond: mp bijection for boolean action}
		for all $g\in G$, the map $\alpha(g,\cdot):X\to X$ is measure-preserving  bijection of $(X,\lambda)$. (in particular, using $g=h=e_G$ in condition \eqref{cond: main condition for boolean action}, we deduce that $\alpha(e_G,x)=x$ for $\lambda$-almost all $x\in X$.)
	\end{enumerate}
\end{defi}

As is customary, we often write $\alpha(g)$ for the map $\alpha(g,\cdot): X\to X$ when $\alpha$
is either a boolean or a spatial measure-preserving action.

\begin{rem} 
	By definition, Condition \eqref{cond: main condition for boolean action} means that for all $g,h\in G$ the following subset is conull: 
	\[
	X_{g,h}= \{x\in X\colon \alpha(gh,x)=\alpha(g,\alpha(h,x))\},
	\]
	but $X_{g,h}$ does depend on $g$ and $h$, in particular there does not need to be a conull subset $X_0\subseteq X$ contained in all the $X_{g,h}$.
\end{rem}

By definition, every measure-preserving spatial action is a measure-preserving boolean action. Spatial and boolean actions admit different natural notions of isomorphism.

\begin{defi}\label{def: spatial iso}
	Given measure-preserving spatial actions $\alpha$ on $(X,\lambda)$ 
    and $\beta$ on $(Y,\eta)$ of a Polish group $G$, 
    a \textbf{spatial isomorphism} between them is a measure preserving injection 
	$\Phi : X_0 \to Y$, where $X_0$ is a conull $\alpha(G)$-invariant Borel subset 
    of $X$, such that for all $g\in G$, and all $x$ in $X_0$, we have 
	\[
	\Phi(\alpha(g,x)) = \beta(g,\Phi(x)).
	\]
    When there exists such a spatial isomorphism, 
    we say that $\alpha$ and $\beta$ are \textbf{spatially isomorphic}.
\end{defi}

\begin{rem}
    Since every Borel injection between standard Borel spaces is a Borel
    isomorphism onto its image \cite[Thm.~4.12]{kechrisClassicalDescriptiveSet1995}, 
    in the above definition the set $Y_0\coloneqq \Phi(X_0)$
    is a Borel subset of $Y$ which is conull since $\Phi$ is measure-preserving,
    and we conclude that being spatially isomorphic is a symmetric relation.
    It is not hard to check that it is also transitive, and hence an equivalence relation.
\end{rem}

\begin{defi}
	Given measure-preserving boolean actions $\alpha$ on $(X,\lambda)$ and $\beta$ on 
    $(Y,\eta)$ of a Polish group $G$, a \textbf{boolean isomorphism} between them is 
    a measure preserving Borel injection
	$\Phi : X_0 \to Y$, where $X_0$ is a full measure Borel subset of $X$, 
    such that for any $g$ in $G$, there is a conull Borel subset $X_g\subseteq X_0$ 
    such that for all $x\in X_g$, 
	\[
	\Phi(\alpha(g,x)) = \beta(g,\Phi(x)).
	\]
\end{defi}

When two boolean actions as above admit a boolean isomorphism between them,
they are called \textbf{booleanly isomorphic}, and it is straightforward
to check that this defines an equivalence relation on boolean actions.
Let us point out that the above definition can be reinforced as follows: we can
require $X_0$ to be equal to $X$ and $\Phi$ to be bijective 
(see Prop.~\ref{prop: boolean iso stronger}).

Observe that any spatial isomorphism between spatial actions is also a boolean isomorphism. 
Two natural problems arise from this for a fixed Polish group $G$
\begin{itemize}
	\item the \textbf{realization problem}: given a boolean $G$-action, does it admit a \emph{spatial realization} (or a \emph{spatial model}), i.e.~is it booleanly isomorphic to a spatial $G$-action? 
	\item the \textbf{boolean to spatial isomorphism problem}: given a boolean isomorphism between spatial $G$-actions, is this isomorphism spatial?
\end{itemize}

When the answers to both theses questions are positive, note that the spatial realization of a boolean action is unique up to spatial isomorphism. Using the fact that any countable intersection of conull sets is conull, it is not hard to see that when $G$ is countable discrete, all its boolean actions are booleanly isomorphic to spatial actions, and all boolean isomorphisms of $G$-actions are spatial.
More generally, a fundamental theorem of Mackey (see \cite{mackeyPointRealizationsTransformation1962}) states that this is still true when the acting group $G$ is locally compact Polish. In our restricted context of measure-preserving actions, the statement is the following.
\begin{thm}[{\textup{\cite[Thm.~1 and Thm.~2]{mackeyPointRealizationsTransformation1962}}}{\label{thm: mackey}}]
    Let $G$ be a locally compact Polish group. 
    \begin{enumerate}[(1)]
        \item  \label{item:Mackey1} Let $\alpha : G \times X \to X$ be a boolean measure-preserving $G$-action on a standard $\sigma$-finite space $(X,\lambda)$. There exists a standard $\sigma$-finite space $(Y,\eta)$ and a spatial $G$-action $\beta$ on $(Y,\eta)$ such that $\alpha$ and $\beta$ are booleanly isomorphic.
        \item  \label{item:Mackey2} Let $\alpha$ and $\beta$ be two spatial measure-preserving $G$-actions on standard $\sigma$-finite spaces $(X,\lambda)$ and $(Y,\eta)$ respectively. Any boolean isomorphism $\Phi$ between $\alpha$ and $\beta$ 
        coincides with a spatial isomorphism on a $G$-invariant conull Borel subset.
    \end{enumerate}
\end{thm}

 However, for non locally compact Polish groups, there might exist spatial actions which are 
 booleanly isomorphic but not spatially. 
 This is the case for $G=\mathfrak S_\infty$, 
 and we will discuss this more thoroughly in Section \ref{sec: two realizations}.

 We now go back to the definition of boolean isomorphism
 between actions, which can be reinforced via the following well-known proposition.
 
 \begin{prop}\label{prop: boolean iso stronger}
    Two measure-preserving boolean actions $\alpha$ on $(X,\lambda)$ and 
    $\beta$ on $(Y,\eta)$ of a Polish group $G$ are booleanly isomorphic if and only if there 
    is a measure preserving bijection 
    $\Phi : X \to Y$ such that for any $g$ in $G$, 
    there is a conull Borel subset $X_g\subseteq X$ such that for all $x\in X_g$, 
    \[
	\Phi(\alpha(g,x)) = \beta(g,\Phi(x)).
	\] 
 \end{prop}
 \begin{proof}
    By definition, assuming that $\alpha$ and $\beta$ are 
    booleanly isomorphic, we have a measure-preserving injection 
    $\Phi:X_0\to Y$ where $X_0$ is a full measure Borel subset of $X$
    satisfying  our equivariance condition. 
    Observe that it suffices to show that there is 
    a measure-preserving bijection $\tilde \Phi:X\to Y$ which coincides with $\Phi$
    on a conull Borel subset of $X_0$.
    
    To this end, let $A\subseteq X_0$ be a Borel uncountable set of measure zero
    (such a set exists since $\lambda$ is atomless, e.g. take for $A$ 
    the triadic Cantor subset
    of $(X,\lambda)=(\R,\Leb)$).

    Then $\tilde A\coloneqq (X\setminus X_0) \sqcup A$ and $\tilde B\coloneqq (Y\setminus \Phi(X_0))\sqcup \Phi(A)$
    are uncountable Borel subsets of standard Borel spaces, so by 
    \cite[Cor.~13.4 and Thm.~15.6]{kechrisClassicalDescriptiveSet1995} there exists a 
    Borel bijection $f: \tilde A\to\tilde B$, and we simply let 
    \(
    \tilde \Phi=f\sqcup \Phi_{\restriction X_0\setminus A}
    \).
\end{proof}

The next remark
was suggested by the first referee and requires the previous proposition.
It can be summed up as follows: 
 the realization problem actually asks whether we can 
 turn the given boolean action 
 into a spatial action without changing the underlying space. 

 \begin{rem}
Let $\alpha$ be a measure-preserving boolean action of a Polish group $G$ on $(X, \lambda)$
for which the realization problem has a positive answer.
Then \(\alpha\) is booleanly isomorphic to a spatial G-action $\beta$ on $(Y, \eta)$ 
through a measure-preserving bijection $\Phi : X \to Y$ 
provided by the previous proposition: for any $g\in G$, 
there is a conull Borel subset $X_g\subseteq X$ such that for all $x \in  X_g$ :
$\Phi (\alpha (g, x)) = \beta (g, \Phi (x))$.
If we define $\tilde \alpha$ by
$\tilde \alpha (g, x) = \Phi^{-1}\beta (g, \Phi (x))$,
then $\tilde \alpha$ is a spatial action on $(X, \lambda)$ 
such that for any $g \in G$ and all $x \in X_g$ we have
$\tilde \alpha (g, x) = \alpha (g, x)$.
In particular $\tilde \alpha$ and $\alpha$ are booleanly isomorphic 
(through the identity map). 
This observation allows us to see the problem of spatial realization 
as finding an appropriate modification of $\alpha$ on the same space.
\end{rem}

We finally observe that the definition of spatial isomorphism could naturally 
be weakened by requiring the $\alpha(G)$-invariant set $X_0$ to be Lebesgue measurable. 
This notion still implies boolean isomorphism of the corresponding 
boolean actions, 
so for locally compact Polish groups this does not change the definition of
being spatially isomorphic thanks to Mackey's Theorem \ref{thm: mackey}.
However, we don't know what the situation is for general Polish groups.

\subsection{Characterization of boolean actions } \label{sec: characterization of action}

 Following the introduction of \cite{glasnerAutomorphismGroupGaussian2005}, we now provide other ways of understanding boolean actions, one of which will be used in our proof of Theorem \ref{thm intro: spatial conjugacy to nice action}.
 We first need to describe the Polish group topology of the group of measure-preserving bijections. 

\begin{defi}
	Let $(X,\lambda)$ be a standard $\sigma$-finite space.
    We denote by $\Aut(X,\lambda)$ the group of measure-preserving bijections, where two such bijections are identified if they coincide on a conull Borel subset of $X$.

The group $\Aut(X,\lambda)$ is endowed with \textbf{weak topology}, defined as follows: $T_n \rightarrow T$ if and only if for all $A \subseteq Y$ Borel subset of finite measure, one has $\lambda(T_n(A) \Delta T(A)) \rightarrow 0$.
    \end{defi}
In order to see that $\Aut(X,\lambda)$ is a Polish group, it is useful to recall the following definition.

    \begin{defi}
The \textbf{measure algebra} $\MAlg_f(X,\lambda)$ is comprised of the Borel subsets of $(X,\lambda)$ of finite measure, where two such subsets are identified if the measure of their symmetric difference is equal to zero. It is equipped with the metric $d_\lambda$, defined as follows~:
\[
d_\lambda(A,B) \coloneqq \lambda(A\Delta B).
\]        
    \end{defi}
    
	\begin{prop}[{\textup{{\cite[Prop.~2.2]{lemaitrePolishTopologiesGroups2022}}}}]
    {\label{prop: Aut Polish}}
$\Aut(X,\lambda)$ is equal to the group of isometries of $(\MAlg_f(X,\lambda),d_\lambda)$ which fix $\emptyset$. As such, it is a Polish group when endowed with the weak topology. 
	\end{prop}

The following proposition tells us that boolean $G$-actions on $(X,\lambda)$ and continuous group homomorphisms $G\to\Aut(X,\lambda)$ are basically the same thing. 
This result is proved in the introduction of \cite{glasnerAutomorphismGroupGaussian2005} in the finite measure case, and it extends almost verbatim to our context. However, their \emph{near actions} have a small difference with our boolean actions: we require $\alpha(g,\cdot)$ to be a measure-preserving bijection while they only require that it is a measure-preserving map, and need an additional axiom ensuring that in the end that it is a bijection \emph{up to a null set}. Having a more restrictive definition which requires $\alpha(g,\cdot)$ to be a bijection of the whole space, we will need to work a little harder in order to prove the corresponding lifting property (item \eqref{item: lifting hom} below).

\begin{prop}{\label{prop: boolean action lift}}
Let $G$ be a Polish group, and $(X,\lambda)$ a standard $\sigma$-finite space. The following hold:
\begin{enumerate}[(I)]
	\item \label{item: boolean yields hom}Every boolean measure-preserving $G$-action $\alpha$  on $(X,\lambda)$  yields a continuous group homomorphism $\pi_\alpha:G\to \Aut(X,\lambda)$.
	\item \label{item: boolean with same hom}Given two boolean measure-preserving $G$-actions $\alpha$, $\beta$ on $(X,\lambda)$, we have $\pi_ \alpha=\pi_\beta$ if and only if the identity map is a boolean isomorphism between $\alpha$ and $\beta$.
	\item \label{item: lifting hom}(lifting property) For every continuous group homomorphism $\pi:G\to \Aut(X,\lambda)$, there is a boolean measure-preserving $G$-action $\alpha$ on $(X,\lambda)$ such that $\pi=\pi_\alpha$. Furthermore, such a lift $\alpha$ is unique up to boolean isomorphism.
\end{enumerate}
\end{prop}
    \begin{proof}
(I): Given a boolean measure-preserving action $\alpha$ of $G$ on $(X,\lambda)$, we can consider the natural mapping 
\[
\begin{array}{rcc}
      \pi_\alpha: G & \longrightarrow & \Aut(X,\lambda) \\
      g & \longmapsto & \alpha(\,g \, , \, \cdot \,)
\end{array}
\]
which is well-defined by Condition \eqref{cond: mp bijection for boolean action} 
from the definition of boolean action, and is a group homomorphism 
by Condition \eqref{cond: main condition for boolean action}.
Recall that any Borel group homorphism between Polish groups is automatically continuous 
by a result which dates back to Banach 
(see e.g.\  \cite[Thm.~2.2]{rosendalAutomaticContinuityGroup2009}).
We are thus left with showing that $\pi_\alpha$ is a Borel map. 
By definition of the weak topology on $\Aut(X,\mu)$ it is enough to show that 
for every Borel finite measure subset $A$ of $X$ and every $\varepsilon>0$, the set 
\[
\mathcal B\coloneqq\{g\in G:\lambda(\pi_\alpha(g)(A)\bigtriangleup A)<\varepsilon\}
\] 
is Borel. 
By the definition of $\pi_\alpha$, we can rewrite $\mathcal B$ as 
$\mathcal B=\{g\in G:\mu(\alpha(g,A)\bigtriangleup A)<\varepsilon\}$.
Since the boolean action $\alpha$ is a Borel map, 
the subset $\Gamma \coloneqq \{(g,x)\in G\times X:\ x\in \alpha(g,A)\}$ is Borel 
and hence $\Gamma_A \coloneqq \Gamma\bigtriangleup (G\times A)$ is also Borel. 
By the Fubini-Tonelli theorem, this implies that the map 
\[
M:g\longmapsto \lambda\left(\left\lbrace x\in X:\ x\in A\bigtriangleup \alpha(g,A)\right\rbrace\right)
\]
which associates to $g\in G$ the measure of the horizontal section of 
$\Gamma_A$ above $g$ is also Borel. 
So we can conclude that $\mathcal B$ is Borel by noting that $\mathcal B=M^{-1}([0,\varepsilon[)$,
 thus finishing the proof that $\pi_\alpha$ is Borel and hence continuous as desired.

(II): This is a straightforward consequence of the definition of $\Aut(X,\lambda)$, which identifies measure-preserving bijections which coincide on a conull set.

(III): 
We first observe that it suffices to show that the identity group homomorphism
\(\id_{\Aut(X,\lambda)}:\Aut(X,\lambda)\to\Aut(X,\lambda)\) can be lifted to a boolean measure preserving \(\Aut(X,\lambda)\)-action 
\(\beta\) on \((X,\lambda)\), namely there is a Borel map \(\beta:\Aut(X,\lambda)\times X\to X\) such that for all \(T\in \Aut(X,\lambda)\), the map \(\beta(T,\cdot)\) is a Borel measure-preserving bijection of \((X,\lambda)\) whose equivalence class up
to a null set is equal to \(T\) (in other words, \(\id_{\Aut(X,\lambda)}=\pi_{\beta}\)).

Indeed, granted that such a boolean action \(\beta\) exists, and given 
an arbitrary continuous group homomorphism $\pi : G \to \Aut(X,\lambda)$, we simply 
let \(\alpha(g,x)=\beta(\pi(g),x)\). By construction we have 
\(\pi=\pi_\alpha\), and the uniqueness
of \(\alpha\) is then a direct consequence of \eqref{item: boolean with same hom}.

So let us prove (III) in the case $G = \Aut(X,\lambda)$ and 
\(\pi=\id_{\Aut(X,\lambda)}\).
As explained in Section \ref{sec: prelim action} we may as well assume that $X$ is an interval of $\R$ and $\lambda$ is the Lebesgue measure restricted to that interval.
Denote by $\LL^0(X,\lambda,X)$ the set of measurable functions from $X$ to itself, up to equality $\lambda$-a.e. 
Its Borel structure is defined by requiring the functions $f \mapsto \lambda(A \cap f^{-1}(B))$ to be measurable, for any Borel subsets $A$ and $B$ of $X$ with $A$ of finite measure. 
This makes $\LL^0(X,\lambda,X)$ a standard Borel space, and it is straightforward to check that $\Aut(X,\lambda)$ is a Borel subset of $\LL^0(X,\lambda,X)$.

Following Glasner-Tsirelson-Weiss, we can now define a Borel way to  lift any element in $\LL^0(X,\lambda,X)$ to a genuine function $X\to X$. Define $V : \LL^0(X,\lambda,X) \times X \longrightarrow X $ by
\[
V(f,x) = \limsup_{\varepsilon \rightarrow 0} \dfrac{1}{2 \varepsilon} \int_{x-\varepsilon}^{x+\varepsilon} f(z)d\lambda(z).
\]
Applying the functional version of Lebesgue's density theorem (see \cite[Thm.~223A.]{fremlinMeasureTheoryVol2-2003}), for any $f$ in $\LL^0(X,\lambda,X)$, the function $x \mapsto V(f,x)$ is in the same $\lambda$-a.e. equality equivalence class as $f$. 
Restricting \(V\) to \(\Aut(X,\lambda)\subseteq \LL^0(X,\lambda,X)\), 
we obtain a Borel map $\beta_0:\Aut(X,\lambda) \times X \to X$
verifying all the required conditions, except that
for all \(T\in\Aut(X,\lambda)\), the map
$\beta_0(T)=V(T,\cdot):X\to X$ is a measure-preserving bijection \emph{up to a null set}. 

In order to correct this, let us first define $\beta_1(T)$ as the partial Borel map 
whose graph is the intersection of the graph of $\beta_0(T,\cdot)$ 
with the inverse of the graph of  $\beta_0(T\inv,\cdot)$. In other words,
\[
\beta_1(T,x)=\beta_0(T,x)\text{ if }\beta_0(T\inv,\beta_0(T,x))=x,
\]
otherwise $\beta_1(T,x)$ is undefined.

Then $\beta_1(T)$ is a measure-preserving bijection between full measure subsets of $X$, 
and we denote by $D_n^T$ the conull subset where $\beta_1(T)^n$ is defined, for $n\in\Z$. 
We finally define $\beta(T)$ by $\beta(T,x)=\beta_1(T,x)$ on $\bigcap_n D_n^T$, 
and by $\beta(T,x) =x$ elsewhere. By construction $\beta(T)$ is now a bijection of $X$ 
itself, 
and it preserves the measure since it coincides with $\beta_1(T)$ on the conull
Borel set $\bigcap_n D_n^T$ .

The proof that $\beta$ is still a Borel map boils down to the fact that for every $n\in\Z$,
the set of $(T,x)\in \Aut(X,\lambda)\times X$ such that 
$x \in X\setminus D_n^T$ is Borel, which we leave to the reader. 
Since for every $T\in\Aut(X,\lambda)$, the bijection $\beta(T)$ coincides with $\beta_0(T)$ 
on a full measure set, the map $\beta$ now satisfies all the axioms of a boolean action 
and $\id_{\Aut(X,\lambda)}=\pi_\beta$ as desired. 
\end{proof}

By the above proposition, boolean measure-preserving $G$-actions are essentially the same thing as continuous group homomorphisms $G\to \Aut(X,\lambda)$. 
We can now justify the terminology properly: the space $\MAlg_f(X,\lambda)$ of finite measure Borel subsets up to measure zero can be equipped with the usual set theoretic operations $\Delta$ and $\bigcap$ so as to become a (non-unital) boolean ring.
This boolean ring, when endowed additionally with the metric $d_\lambda$, has its automorphism group equal to $\Aut(X,\lambda)$ as a consequence of Proposition \ref{prop: Aut Polish}.

So continuous group homomorphisms $G\to \Aut(X,\lambda)$ are exactly continuous $G$-actions by automorphism on the metric boolean ring $(\MAlg_f(X,\lambda),\bigtriangleup,\cap,d_\lambda)$, and since these homomorphims are the same thing as boolean actions by the above proposition, the terminology is now fully justified. 

As we will see in the next section, we can also view boolean actions as 
 trace preserving actions on the von Neumann algebra $\LL^\infty(X,\lambda)$
 satisfying an additional continuity condition.

\subsection{Actions on function spaces}\label{sec: actions on function spaces}

Given a standard $\sigma$-finite space $(X,\lambda)$, we denote by $\LL^0(X,\lambda)$ the space of measurable functions $X\to \C$, two such functions being identified if they coincide on a conull set. The following subspaces of functions are of importance to us:
\begin{itemize}
	\item  the space $\LL^\infty(X,\lambda)$ of complex-valued bounded measurable functions on $X$, endowed with the essential supremum norm $\norm{\cdot}_\infty$;
	\item  the space $\LL^1(X,\lambda)$ of complex-valued integrable functions on $X$, endowed with the $\LL^1$ norm given by $\norm{f}_1=\int_X \abs f d\lambda$;
	\item  the space $\LL^2(X,\lambda)$of complex-valued square integrable functions on $X$, endowed with the Hilbert space structure given by the scalar product
	$\langle f,g\rangle=\int_X\overline{f} g d\lambda$.
\end{itemize}

We have a natural left action of $\Aut(X,\lambda)$ on $\LL^0(X,\lambda)$ given by
\[
T \cdot f(x) = f(T\inv(x)).
\]
All the subspaces $\LL^\infty(X,\lambda)$, $\LL^2(X,\lambda)$ and $\LL^1(X,\lambda)$ are invariant under this action. Moreover, they are acted upon by isometries (for the last two, this uses the fact that elements of $\Aut(X,\lambda)$ preserves the measure, while for $\LL^\infty(X,\lambda)$ we only need that the measure class is preserved).

However, the $\Aut(X,\lambda)$-action on $\LL^\infty(X,\lambda)$ is not continuous (for the norm $\norm{\cdot}_\infty$), while the $\Aut(X,\lambda)$-action on $\LL^2(X,\lambda)$ and $\LL^1(X,\lambda)$ are both continuous.
Indeed, for any Borel subset $A$ of $(X,\lambda)$ of finite measure, $A$ is sent by an element $T$ of $\Aut(X,\lambda)$ to a subset $T(A)$ which is close to $A$ in $(\MAlg_f(X,\lambda),{d_\lambda})$ whenever $T$ is close to $\id_X$. This suffices to prove the continuity, as both actions are by isometries and linear combinations of finitely many characteristic functions of finite measure subsets are dense in $\LL^{2}(X,\lambda)$ and $\LL^1(X,\lambda)$ for their respective norms.

We denote by $\mathcal{B}(\LL^{2}(X,\lambda))$ the space of bounded operators on $\LL^{2}(X,\lambda)$, that is to say the bounded linear maps from $\LL^{2}(X,\lambda)$ to itself, and by $M_f$ the operator of pointwise multiplication by a element $f$ of $\LL^{\infty}(X,\lambda)$. We recall that the map
\[
\begin{array}{ccccc}
	M & : & \LL^{\infty}(X,\lambda) & \longrightarrow & \mathcal{B}(\LL^{2}(X,\lambda)) \\
	& & f & \longmapsto &  M_f
\end{array}
\]
is an isometric $\ast$-isomorphism between $\LL^{\infty}(X,\lambda)$ (endowed with the algebra structure provided by pointwise multiplication) and a (von Neumann) subalgebra of $\mathcal{B}(\LL^{2}(X,\lambda))$, where $*$ is the involution taking $f\in\LL^\infty$ to its complex conjugate $x\mapsto \overline{f(x)}$.
The map $M$ allows us to endow $\LL^\infty(X,\lambda)$ with the \textbf{strong operator topology}: a net $f_n$ converges to $f$ if and only for all $\xi \in\LL^2(X,\lambda)$, $M_{f_n}\xi\to M_f\xi$ in $\LL^2$ norm.

\begin{rem}
    Suppose $\lambda$ is infinite and $\mu$ is a finite measure equivalent to $\lambda$. Let $f=\frac{d\lambda}{d\mu}$ be the Radon-Nikodym derivative, then the map $g\mapsto \sqrt f g$ induces a surjective isometry $\LL^2(X,\lambda)\to \LL^2(X,\mu)$ which commutes with the action by multiplication of $\LL^\infty(X,\mu)=\LL^\infty(X,\lambda)$. We conclude that strong convergence in $\LL^\infty(X,\mu)$ is an intrinsic notion which does not depend to the choice of a (possibly infinite) $\sigma$-finite measure in the class of $\mu$.
\end{rem}

We finally identify $\Aut(X,\lambda)$ to the following group.

\begin{prop}\label{prop: aut as integral preserving auto}
	Given any standard $\sigma$-finite space $(X,\lambda)$, the group $\Aut(X,\lambda)$ naturally identifies to the group of $*$-automorphisms of $\LL^\infty(X,\lambda)$ which preserve the integral of elements of $\LL^\infty(X,\lambda)\cap\LL^1(X,\lambda)$.
\end{prop}
\begin{proof}
	We have already observed at the beginning of this section that precomposition by the inverse allows to view every measure-preserving bijection $T$ of $(X,\lambda)$ as an integral-preserving automorphism $\alpha_T$ of $\LL^\infty(X,\lambda)$, namely $\alpha_T(f)(x)=f(T\inv(x))$. Observe that for every $A\subseteq X$ Borel, $\alpha_T(\chi_A)=\chi_{T(A)}$. 
	
	Conversely, every integral-preserving $*$-automorphims $\alpha$ of $\LL^\infty(X,\lambda)$
	must take projections to projections. 
	Since the projections which are integrable are naturally identified to characteristic functions of elements of $\MAlg_f(X,\lambda)$, every integral-preserving automorphism $\alpha$ of $\LL^\infty(X,\lambda)$ defines a measure-preserving transformation $T$.
    Moreover, standard results on $C^*$-algebras yield that $\alpha$ is an isometry for $\norm{\cdot}_\infty$ (see e.g.~\cite[Cor.~1.8]{conwayCourseOperatorTheory2000}).
	Using the $\norm{\cdot}_\infty$-density of the linear span of projections, we conclude that $\alpha$ is equal to $\alpha_T$ as wanted.
\end{proof}

Having the above proposition and Proposition \ref{prop: boolean action lift} in mind, 
we can view boolean actions on $(X,\lambda)$ as integral-preserving actions 
by $*$-automorphisms
\footnote{This is not the right point of view if one wants to make sense of continuous 
actions on von Neumann algebras in general. 
In a more general setup, one conveniently uses the fact that any automorphism yields 
an isometry of the predual which completely determines it, 
thus getting a natural Polish topology on the automorphism group of any von Neumann algebra 
with separable predual.} 
on $\LL^\infty(X,\lambda)$ which are \textbf{continuous} in the following sense 
(derived from Proposition \ref{prop: Aut Polish}): 
whenever $g_n\to g$ and $A\in\MAlg_f(X,\lambda)$, we have 
$\int_X\abs{g_n\cdot \chi_A- g\cdot \chi_A} d\lambda \to 0$,
or equivalently whenever $g_n\to e_G$ and $A\in\MAlg_f(X,\lambda)$, we have 
$\int_X\abs{g_n\cdot \chi_A- \chi_A} d\lambda \to 0$. 
Boolean isomorphisms then become the following. 

\begin{prop}\label{prop: chara boolean iso of action}
    Let $\alpha$ and $\beta$ be two boolean $G$-actions on $(X,\lambda)$ and $(Y,\eta)$ respectively, both viewed as  continuous actions by integral-preserving automorphisms of their respective $\LL^\infty$ spaces. 
    Then $\alpha$ and $\beta$ are booleanly isomorphic if and only if there is an integral-preserving $*$-isomorphism $\rho:\LL^\infty(X,\lambda)\to \LL^\infty(Y,\eta)$ such that for all $f\in \LL^\infty(X,\lambda)$, we have
    $\rho(\alpha(g,f))=\beta(g,\rho(f))$.
\end{prop}
\begin{proof}
We may as well assume that $X=Y$ and $\lambda=\eta$, so that by the above proposition $\rho$ can be lifted to a measure-preserving bijection $\Phi:X\to X$. Using the uniqueness of lifts of elements of $\Aut(X,\lambda)$ up to measure zero, it is then straightforward to check that $\Phi$ satisfies the required conditions to be a boolean isomorphism between the actions.
\end{proof}

\subsection{Locally finite measures}{\label{sec: locally finite and Radon}}

In this section, we introduce the main property of a Borel measure on a Polish space that we will be interested in, and connect it with the notion of a Radon measure when the ambient space is in addition locally compact.

\begin{defi}
    Let $\lambda$ be a Borel measure on a Polish space $X$. We say that $\lambda$ is \textbf{locally finite} if every $x\in X$ admits an open neighborhood $U$ such that $\lambda(U)<+\infty$.
\end{defi}

Of course this property is only interesting for infinite measures. Let us first observe that it implies $\sigma$-finiteness.

\begin{lem}\label{lem: locally finite Polish implies sigma finite}
    Every locally finite measure $\lambda$ on a Polish space $X$ is $\sigma$-finite.
\end{lem}
\begin{proof}
    By assumption, $X$ is covered by finite measure open sets. Lindelöf's lemma grants us a countable subcover which witnesses the fact that $\lambda$ is $\sigma$-finite.
\end{proof}
    \begin{rem}
A $\sigma$-finite measure on a standard Borel space needs not be locally finite. For instance, the counting measure on $\Q$ extends to a non locally finite $\sigma$-finite measure on $\R$. More interesting examples will be given later on.
    \end{rem}

We then recall a definition of Radon measures, following \cite[Sec.~7.2]{cohnMeasureTheorySecond2013}.  

    \begin{defi}\label{def: Radon measure}
A measure $\lambda$ on the Borel $\sigma$-algebra of a Hausdorff topological space $X$ is called \textbf{Radon} when it verifies the following:
\begin{enumerate}
    \item $\lambda(K) < + \infty$ for any compact subset $K \subseteq X$,
    \item for each open subset $U$ of $X$, we have $\lambda(U) = \sup\{ \lambda(K) \mid K \subseteq U ,  \mbox{ and $K$ is compact} \}$ (inner regularity on open sets).
    \item for each Borel subset $A$ of $X$, we have $\lambda(A) = \inf\{ \lambda(U) \mid A \subseteq U ,  \mbox{ and $U$ is open} \}$ (outer regularity on Borel sets).
\end{enumerate}
    \end{defi}

 Recall that a locally compact space is Polish if and only if it is second-countable (see \cite[Thm.~5.3]{kechrisClassicalDescriptiveSet1995}). In such spaces, local finiteness behaves particularly nicely.

\begin{prop}{\label{prop: Radon equivalent to locally finite on locally compact}}
    Let $X$ be a locally compact Polish space, let $\lambda$ be a Borel measure on $X$. 
    Then $\lambda$ is Radon if and only if it is locally finite. 
\end{prop}
\begin{proof}
    Suppose $\lambda$ is Radon, then since it is finite on compact sets and $X$ is locally compact, we conclude that $\lambda$ is locally finite.

    Conversely, suppose $\lambda$ is locally finite. Then by compactness every compact subset can be covered by finitely many finite measure open subsets, and hence has finite measure.
    The conclusion now follows from \cite[Prop.~7.2.3]{cohnMeasureTheorySecond2013}.
\end{proof}

    \begin{rem}{\label{innerregonallBorel}}
As noted in \cite[Prop. 7.2.6]{cohnMeasureTheorySecond2013}, $\sigma$-finiteness implies inner regularity on all Borel sets, not only open sets. Putting together the previous proposition and Lemma \ref{lem: locally finite Polish implies sigma finite}, we see that locally finite Borel measures on locally compact Polish spaces are always inner regular on all Borel sets.
    \end{rem}

We finally quote the Riesz-Markov-Kakutani representation theorem, which will allow us to build Radon  measures on new spaces.

\begin{thm}[{Riesz-Markov-Kakutani, see e.g.~\cite[Thm.~2.14]{rudinRealComplexAnalysis1987}}]{\label{thm: Riesz}}
Let $Y$ be a Polish locally compact space, and let $\Psi$ be a positive linear functional on the space of complex-valued compactly supported continuous functions $C_c(Y)$. Then there exists a unique 
Borel measure $\eta$ on $Y$ such that
\[
\Psi(f) = \int_Y f d\eta
\]
holds for each $f$ in $C_c(Y)$. Moreover, the measure $\eta$ is Radon.
\end{thm} 

\begin{rem}
    In the proof of Theorem \ref{thm: admits model is equivalent to functions separate points},
    it is crucial that the uniqueness of 
    \(\eta\) does not require it to be Radon. 
    As observed in the beginnning of the proof of \cite[Thm.~2.14]{rudinRealComplexAnalysis1987} 
    any Borel measure \(m\) such that 
    \(
\Psi(f) = \int_Y f dm
\)
holds for each $f$ in $C_c(Y)$ must be finite on compact sets,
hence locally finite, which implies that it is Radon by 
Proposition \ref{prop: Radon equivalent to locally finite on locally compact}, so the above statement does follow from 
\cite[Thm.~2.14]{rudinRealComplexAnalysis1987} (see also \cite[Thm.~2.18]{rudinRealComplexAnalysis1987}).
\end{rem}

\subsection{Gelfand spaces in the non unital case}

In order to build spatial actions, we will as in \cite{glasnerAutomorphismGroupGaussian2005} crucially use separable $C^*$ algebras, 
but in our setup these will be non-unital. 
We thus start by recalling the Gelfand-Naimark theorem
in the case of a non-unital commutative $C^*$-algebra.

	\begin{defi}
The \textbf{spectrum} of a commutative $C^\ast$-algebra $\mathcal{A}$ is the space of \textbf{characters} on $\mathcal{A}$, that is to say the space of non-zero homomorphisms $\mathcal{A} \longrightarrow \mathbb{C}$. We denote by $\spec(\mathcal A)$ the spectrum of $\mathcal{A}$. It is locally compact and Hausdorff when equipped with the topology of pointwise convergence, and if $\mathcal{A}$ is unital, it is a compact space.
	\end{defi}

	\begin{prop}
		\label{prop: separable implies second-countable spectrum}
If a commutative $C^\ast$-algebra $\mathcal{A}$ is separable, then $\spec(\mathcal A)$ is a locally compact Polish space.
	\end{prop}
    \begin{proof}
        By \cite[Rem.~4.4.1]{murphyAlgebrasOperatorTheory1990}, $\spec(\mathcal A)$ is second-countable. 
        Since every locally compact second-countable Hausdorff space is Polish \cite[Thm.~5.3]{kechrisClassicalDescriptiveSet1995}, the conclusion follows.
    \end{proof}

	\begin{thm}[{Gelfand-Naimark, see e.g.~\cite[Thm.~2.1.10]{murphyAlgebrasOperatorTheory1990}}]\label{thm: Gelfand}
Let $\mathcal{A}$ be a non-zero commutative $C^\ast$-algebra. The Gelfand representation 
\[
\begin{array}{ccc}
\mathcal{A} & \longrightarrow & C_{0}(\spec(\mathcal A) )\\
f & \longmapsto & \widehat{f}
\end{array}
\]
where $\widehat{f}(z) = z(f)$, between $\mathcal{A}$ and the space of continuous functions that vanish at infinity on $\spec(\mathcal A)$ is an isometric $\ast$-isomorphism. 
	\end{thm}

We now give the proof of the following fact, which is probably well-known, but for which we have found no reference in the literature. 
Combined with Theorem \ref{thm: Riesz}, this will allow us to define a Radon measure on $\spec(\mathcal A)$. 
We are grateful to Georges Skandalis for pointing this fact out to us.

	\begin{prop}{\label{prop:FaitGeorges}}
Let $\mathcal{A} = C_0(X)$, where $X$ is a locally compact Polish space. Let also $\mathcal{I}$ be an ideal of $\mathcal{A}$, dense in $\mathcal{A}$ with regards to the sup norm $\norm{\cdot}_{\infty}$. Then, $C_c(X) \subseteq \mathcal{I}$, where $C_c(X)$ denotes the space of compactly supported continuous functions on $Y$.
	\end{prop}

	\begin{proof}
Let us fix a non-zero function $f$ in $C_c(X)$, supported in a compact subset $K$. 

The Tietze extension theorem grants us a function $g$ in $C_0(X)$ such that $g = 1$ on $K$. 
By density, we have a  function $h\in\mathcal{I}$ such that $\norm{g-h}_\infty <\frac 12$. In particular $$\supp{f} \subseteq K \subseteq \left \{x\in X\colon \abs{h(x)} > \frac 12 \right\}.$$ 
Applying the Tietze extension theorem again, we find $j \in C_0(Y)$ such that $j = \frac 1 h$ on $K$. 
We conclude the proof by noting that  $h j f = f$ is in $\mathcal{I}$, as $\mathcal I$ is an ideal.
	\end{proof}

\subsection{Density for \texorpdfstring{$C^*$}{C*} subalgebras of \texorpdfstring{$\LL^\infty$}{Linfinity}}

In what follows, given a set of functions $\mathcal F\subseteq\LL^{\infty}(X,\lambda)$, we denote by $(\mathcal F)_1$ the intersection of $\mathcal F$ with the unit ball of $\LL^\infty(X,\mu)$ for the norm $\norm{\cdot}_\infty$.

\begin{prop}{\label{prop: density for subcstar}}
    Let $\lambda$ be a $\sigma$-finite measure on a standard Borel space $X$.
    Let $\mathcal G$ be a unital $C^*$-subalgebra of $\LL^\infty(X,\lambda)$, i.e. a $\norm\cdot_\infty$-closed unital $*$-subalgebra of $\LL^\infty(X,\lambda)$.
    The following four conditions are equivalent
    \begin{enumerate}[(1)]
    \item \label{item:L2 dense ball}$(\mathcal G)_1\cap\LL^2(X,\lambda)$ is $\norm\cdot_2$-dense in $(\LL^\infty(X,\lambda))_1\cap\LL^2(X,\lambda)$;
    \item \label{item:L2 dense} $\mathcal G\cap\LL^2(X,\lambda)$ is $\norm\cdot_2$-dense in $\LL^2(X,\lambda)$;
    \item \label{item:finite support dense}the set of finite measure supported elements of $\mathcal G$ is $\norm\cdot_2$-dense in $\LL^2(X,\lambda)$;
    \item \label{item:L1 dense} $\mathcal G\cap\LL^1(X,\lambda)$ is $\norm\cdot_1$-dense in $\LL^1(X,\lambda)$.
\end{enumerate}
Moreover, these equivalent conditions imply:
\begin{enumerate}[(1)]\setcounter{enumi}{4}
\item\label{item: strong dense}$\mathcal G$ is strongly dense in $\LL^\infty(X,\lambda)$. 
\end{enumerate}
Finally, if $\lambda$ is finite, \eqref{item: strong dense} implies the four above conditions.
\end{prop}
\begin{proof}
    We clearly have that \eqref{item:finite support dense} implies \eqref{item:L2 dense}. To see the converse, it suffices to show that finite measure supported elements of $\mathcal G$ are $\norm\cdot_2$-dense in $\mathcal G\cap \LL^2(X,\lambda)$.

    To this end, let $f\in \mathcal G\cap \LL^2(X,\lambda)$. For every $\delta>0$ and $z\in\C$, let 
    \[
    p_\delta(z)=
    \left\{
\begin{array}{ll}
0 & \mbox{if } \abs z \leqslant \delta,\\
z \times \frac{\abs z - \delta}{\abs z} & \mbox{otherwise.}\\\end{array}
\right.
    \]
    Since $\mathcal G$ is a $C^*$-algebra and $p_\delta$ is continuous, we have $p_\delta\circ f\in\mathcal G$. Moreover, $p_\delta\circ f$ has finite measure support because $f\in\LL^2(X,\lambda)$, and the Lebesgue dominated convergence theorem ensures us that $$\lim_{\delta\to 0}p_\delta\circ f=f$$ in the $\LL^2$-norm, so we conclude that the set of finite measure supported elements of $\mathcal G$ is $\norm\cdot_2$-dense in $\mathcal G\cap \LL^2(X,\lambda)$, which shows that \eqref{item:L2 dense} implies \eqref{item:finite support dense} as wanted.

    Now since bounded elements are dense in $\LL^2(X,\lambda)$, it is clear that \eqref{item:L2 dense ball} implies \eqref{item:L2 dense}. 
    Conversely, if $\mathcal G\cap \LL^2(X,\lambda)$ is $\norm\cdot_2$ dense in $\LL^2(X,\lambda)$, let $f\in (\LL^\infty(X,\lambda))_1 \cap \LL^2(X,\lambda)$. Let $f_n\to f$ in $\norm\cdot _2$ with $f_n\in\mathcal G$. Define the continuous function $q:\C\to\C$ by
    \[
    q(z)=
    \left\{
\begin{array}{ll}
z & \mbox{if } \abs z \leqslant 1,\\
\frac{z}{\abs z}  & \mbox{otherwise.}\\\end{array}
\right.
    \]
    Since $\mathcal G$ is a $C^*$-algebra, for all $n\in\N$ we have $q\circ f_n\in\mathcal (\mathcal G)_1$.
    Moreover for all $x\in X$ we have $\abs {q\circ f_n(x)-f(x)}\leq \abs{f_n(x)-f(x)}$, so we also have $qf_n\to f$ in $\norm\cdot_2$, thus finishing the proof that \eqref{item:L2 dense} implies \eqref{item:L2 dense ball}. So conditions \eqref{item:L2 dense ball}, \eqref{item:L2 dense} and \eqref{item:finite support dense} are all equivalent.

    We now prove that \eqref{item:L1 dense} and \eqref{item:L2 dense} are equivalent using the square and root functions on complex moduli:
    \[
    \begin{array}{l}
         r(z) = r(\abs{z}e^{i\theta_z} ) \coloneqq \sqrt{\abs{z}}e^{i\theta_z}  \\
         s(z) = s(\abs{z}e^{i\theta_z} ) \coloneqq \abs{z}^2e^{i\theta_z} .
    \end{array}
    \]
    Note that $r$ and $s$ are homeomorphisms of $\C$, inverse of each other. Moreover given $f\in\LL^{1}(X,\lambda)$ the function $s\circ f$ is in $\LL^2(X,\lambda)$ and vice-versa. Furthermore, if $f$ is in $\mathcal{G}$, $r \circ f$ and $s \circ f$ remain in $\mathcal{G}$ as it is a $C^*$ algebra. We prove that \eqref{item:L1 dense} implies \eqref{item:L2 dense}.
    
    Let then $g\in\LL^{2}(X,\lambda)$, and let $f=s\circ g$ so that $g= r \circ f$ with $f \in \LL^{1}(X,\lambda)$. By assumption, we have a sequence $(f_n)$  converging to $f$ for the $\LL^1$ norm, with $f\in\mathcal G$.
    By a classical theorem attributed to Riesz–Fischer (see e.g.~\cite[IV \S3 Thm. 3]{BourbakiIntegration}), we can extract a subsequence $(f_{n_k})$ converging pointwise almost everywhere to $f$ so that there exists $h \in \LL^{1}(X,\lambda)$ verifying $\abs{f_{n_k}} \leqslant h$ almost everywhere for any $k$. Thus, the sequence $(r \circ f_{n_k})$ is in $\LL^{2}$ and $r \circ f_{n_k} \to r \circ f$ pointwise almost everywhere. Moreover, we have $\abs{r \circ f_{n_k}} \leqslant r \circ h\in \LL^2(X,\lambda)$ for any $k$. By the $\LL^2$ version of the Lebesgue dominated convergence theorem, we have $r\circ f_{n_k}\to r\circ f$ in $\LL^2$ norm.
    In other words $g=r \circ f$ can be approximated for $\norm{\cdot}_2$ by functions in $\mathcal G\cap\LL^2(X,\lambda)$ as wanted, thus showing that \eqref{item:L1 dense} implies \eqref{item:L2 dense}. The symmetric argument gives the reverse implication, so conditions \eqref{item:L2 dense ball}, \eqref{item:L2 dense},   \eqref{item:finite support dense} and \eqref{item:L1 dense} are all equivalent.
    
     We now connect them to \eqref{item: strong dense} by showing first that \eqref{item:L2 dense ball} implies strong density of $\mathcal G$ in $\LL^\infty(X,\lambda)$. 
    To this end, first note that $\LL^\infty(X,\lambda)\cap\LL^2(X,\lambda)$ is strongly dense in $\LL^\infty(X,\lambda)$: if $f\in\LL^\infty(X,\lambda)$ and $(X_n)$ is an increasing sequence of subsets of finite measure subsets of $X$ such that $X=\bigcup_{n\in\N}X_n$, then for all $\xi\in\LL^2(X,\lambda)$
    \[
    \norm{f\xi-f\mathds{1}_{X_n}\xi}_2^2=\int_{X\setminus X_n}\abs{f(x)\xi(x)}^2d\lambda(x) \longrightarrow 0
    \]
    by the Lebesgue dominated convergence theorem, so $f\mathds{1}_{X_n}\to f$ strongly.
    Towards showing the desired implication, assume \eqref{item:L2 dense ball}, 
    it now suffices to strongly approximate  any $f\in\LL^\infty(X,\lambda)\cap\LL^2(X,\lambda)$ by a sequence of elements of $\mathcal G$.
    Replacing $f$ by $f/\norm{f}_\infty$ if need be, we may as well assume $f\in (\LL^\infty(X,\lambda))_1$. We then have $a_n\to f$ in $\norm\cdot_2$ with $a_n\in (\mathcal{G})_1$. By density of step functions and the fact that $\norm{a_n}_\infty \leq 1$, it suffices to show that if $B$ is a finite measure subset, then $a_n\mathds{1}_B\to f\mathds{1}_B$ for $\norm{\cdot}_2$, which is an immediate consequence of the fact that 
    \[
    \norm{a_n\mathds{1}_B-f\mathds{1}_B}_2^2=\int_B\abs{a_n(x)-f_n(x)}^2d\lambda(x) \leq\norm{a_n-f}_2^2\longrightarrow 0.
    \]

    Finally, if $\lambda$ is finite, let us show \eqref{item: strong dense} implies \eqref{item:L2 dense}. Take any generalized sequence $(f_i)$ in $\mathcal G$ converging strongly to $f\in\LL^\infty(X,\lambda)$, then $f_i \mathds{1}_X\to f\mathds{1}_X$, which means exactly that $\norm{f_i-f}_2\to 0$. This concludes the proof, as the finiteness of $\lambda$ yields that $\LL^\infty(X,\lambda)$ is $\norm{\cdot}_2$-dense in $\LL^2(X,\lambda)$.
\end{proof}
\begin{rem}
    Similar arguments show that we can add to the list of the first 4 equivalent conditions the $\LL^1$ versions of condition \eqref{item:L2 dense ball} and \eqref{item:finite support dense}, namely
\begin{itemize}
    \item[(1')] $(\mathcal G)_1\cap\LL^1(X,\lambda)$ is $\norm\cdot_1$-dense in $(\LL^\infty(X,\lambda))_1\cap\LL^1(X,\lambda)$;
    \item[(3')] the set of finite measure supported elements of $\mathcal G$ is $\norm\cdot_1$-dense in $\LL^1(X,\lambda)$.
\end{itemize}
    
     We did not include them in the statement for brevity.
\end{rem}
As a concrete example of the above situation, we have the following well-known lemma, which will be useful in our characterization of spatial realizations of Boolean actions via Gelfand's theorem.

\begin{lem}[{see \cite[Prop.~7.4.3]{cohnMeasureTheorySecond2013}}]
\label{lem: dense C0 in L2}
Let $Y$ be a locally compact Hausdorff space equipped with a Radon measure $\eta$, and consider the $C^*$-algebra $C_{0}(Y)$ of continuous functions that vanish at infinity on $Y$. 
Then $C_{0}(Y)\cap \LL^2(Y,\eta)$ is $\norm\cdot_2$-dense in $\LL^{2}(Y,\eta)$. 
In particular, the strong closure of $C_{0}(Y)$ in $\mathcal B(\LL^2(Y,\eta))$ is equal to $\LL^{\infty}(Y,\eta)$.\qed
\end{lem}

\section{\texorpdfstring{$G$}{G}-continuity and spatial models}{\label{sec: G continuity and char of existence spatial model}}

\subsection{\texorpdfstring{$G$}{G}-continuity}

We recall the definition of $G$-continuity, which has notably been used in \cite{kwiatkowskaSpatialModelsBoolean2011} 
and in \cite{glasnerSpatialNonspatialActions2005} to discuss the existence of spatial models for boolean actions of Polish groups.

\begin{defi}
Let $G$ be a Polish group, and let $\alpha$ be a measure-preserving boolean $G$-action on a standard $\sigma$-finite space $(X,\lambda)$. We say that $f$ in $\LL^{\infty}(X,\lambda)$ is \textbf{$G$-continuous} if $\norm{f-f\circ\alpha(g_n\inv,\cdot)}_\infty\to 0$ whenever $g_n \to e_G$.
The space of $G$-continuous functions will thereafter be denoted by $\mathcal{G}$.
\end{defi}

  \begin{rem}{\label{Gcontnonbounded}}
Although $G$-continuity is defined for functions in $\LL^{\infty}(X,\lambda)$, the previous definition extends to non-essentially bounded functions. Indeed, only the difference of $f\circ\alpha(g_n\inv , \cdot)$ and $f$ needs to be in $\LL^{\infty}(X,\lambda)$, when $g$ is close enough to $e_G$.
This remark will be useful in the proof of Proposition \ref{prop: dense Gcont for iso of lc}.
    \end{rem}

The set $\mathcal{G}$ of $G$-continuous functions is easily seen to be a unital $\ast$-subalgebra of $\LL^{\infty}(X,\lambda)$. 
We will now see that it is actually a $C^*$-algebra, i.e. that it is $\norm{\cdot}_\infty$-closed. 
This will be a direct consequence of the following well-known proposition.

\begin{prop}
    Let $(M,d)$ be a metric space on which a topological group $G$ acts by isometries. Denote by $M_G$ the set of $x\in M$ such that  $d(g_n\cdot x,x)\to 0$ whenever $g_n \to e_G$. Then $M_G$ is a closed subset of $M$, and it is the largest subset of $M$ onto which $G$ acts continuously.
\end{prop}
\begin{proof}
    Let $(x_k)$ be a sequence of elements of $M_G$ converging to $x\in M$. Fix $\varepsilon > 0$. Towards showing $x\in M_G$, take a sequence $(g_n)$ tending to $e_G$. By the triangle inequality, we have the following for all $k,n\in\N$:
    \[
    d(x,g_n \cdot x) \leqslant d(x,x_k) + d(x_k,g_n \cdot x_k) + d(g_n \cdot x_k, g_n \cdot x).
    \]
    Fix $k$ large enough so that $d(x,x_k)<\varepsilon$.
    Since $g_n$ is an isometry, we have $d(g_n \cdot x_k, g_n \cdot x)= d(x,x_k)$, and so both the first and third terms in the above sum are less than $\varepsilon$. For the second term, since $x_k$ is a point of continuity for the $G$-action, it is smaller than $\varepsilon$ for $n$ big enough, which concludes the proof that $M_G$ is closed. 
    
 	Now by the definition of $M_G$, the restriction of the action of $G$ on a set which intersects the complement of $M_G$ cannot be continuous. We thus only need to show that the action on $M_G$ is continuous. To this end, let $x_n\to x$ with $x_n,x\in M_G$ and let $g_n\to g$.
 	Then 
 \begin{align*}
 	d(g_n\cdot x_n,g\cdot x) 
 	& \leq d(g_n\cdot x_n,g_n\cdot x)+d(g_n\cdot x, g\cdot x)\\
 	& = d(x_n,x)+d(g\inv g_n\cdot x,x).
 \end{align*}
Since $g\inv g_n\to e_G$ and $x\in M_G$, we have $d(g\inv g_n\cdot x,x)\to 0$, and since $x_n\to x$, we also have $d(x_n,x)\to 0$. So $d(g_n\cdot x_n,g\cdot x)\to 0$ as wanted, which finishes the proof.
\end{proof}
\begin{cor}\label{cor: curly G is a Cstar algebra}
	Let $\pi: G\to \Aut(X,\lambda)$ be a boolean action. 
	Then the space $\mathcal G$ of $G$-continuous functions is actually a unital $C^*$-subalgebra of $\LL^\infty(X,\lambda)$ onto which $G$ acts continuously.
\end{cor}
\begin{proof}
	We have already observed that the $G$-action on $(M,d)=(\LL^\infty(X,\lambda),\norm{\cdot}_\infty)$ associated to $\pi$ (given by $g\cdot f=f\circ \pi(g)\inv$) is by isometries, and by definition $\mathcal G=M_G$ so the conclusion directly follows from the previous proposition.
\end{proof}

We finally observe that we have the following natural source of $G$-continuous functions.

\begin{lem}\label{lem: G continuous from continuous action on lc}
    Let $X$ be a locally compact Polish space, let $\alpha: G\times X\to X$ be a continuous $G$-action, consider the action on $C_0(X)$ by precomposition by the inverse, then every element of $C_0(X)$ is $G$-continuous, namely $\norm{f-f\circ\alpha(g\inv,\cdot)}_\infty\to 0$ when $g\to e_G$, where $\norm{\cdot}_\infty$ is the supremum norm. 
\end{lem}
\begin{proof}   
    Let $f\in C_0(X)$, let $\varepsilon>0$. Let $K$ be a compact set such that for all $x\notin K$ we have $\abs{f(x)}<\varepsilon$. 
    By continuity of the action and of $f$, for every $x\in K$ there is a neighborhood $U_x$ of $x$ and a symmetric neighborhood $\mathcal{N}_x$ of $e_G$ such that $\abs{f(\alpha(g,x'))-f(x)}<\varepsilon$ for all $g\in \mathcal{N}_x$ and all $x'\in U_x$. 
    Take a finite subcover $(U_{x_i})_{i=1}^n$ of $K$, let $\mathcal{N}=\bigcap_{i=1}^n \mathcal{N}_{x_i}$, then if $x\in U_{x_i}$, the triangle inequality yields
    \begin{align*} 
    \abs{f(\alpha(g,x))-f(x)}\leq \abs{f(\alpha(g,x))-f(x_i)}+\abs{f(x_i)-f(x)}<2\varepsilon.
    \end{align*}
    Now let $x\in X$ be arbitrary and take $g\in \mathcal{N}$, we have three cases to consider:
    \begin{itemize}
    \item if $x\in K$ then  $\abs{f(\alpha(g,x)))-f(x)}<2\varepsilon$ by what we just did;
    \item if $x\notin K$ and $\alpha(g,x)\in K$ by symmetry of $\mathcal{N}$ $\abs{f(\alpha(g\inv,\alpha(g,x))-f(\alpha(g,x))}<2\varepsilon$ so $\abs{f(x)-f(\alpha(g,x))}<2\varepsilon$;
    \item if $x\notin K$ and $\alpha(g,x)\notin K$, then we both have $\abs{f(x)} <\epsilon$ and $\abs{f(\alpha(g,x))}<\varepsilon$, so again $\abs{f(x)-f(\alpha(g,x))}<2\varepsilon$.
    \end{itemize}
    We conclude that $\norm{ f-f\circ\alpha(g,\cdot)}_\infty<2\varepsilon$ for all $g\in \mathcal{N}$, which finishes the proof that $f$ is $G$-continuous.
\end{proof}

\subsection{Continuous Radon models for infinite measure-preserving boolean actions}

Following the terminology of Glasner-Tsirelson-Weiss, a \textbf{continuous spatial model}
for a boolean measure-preserving action $\alpha$ of a Polish group $G$ on $(X,\lambda)$ 
is a continuous  $G$-action on a Polish measured space $(Y,\eta)$ 
which is booleanly isomorphic to $\alpha$. 
Moreover, if $Y$ is a locally compact Polish space and $\eta$ is Radon, 
we call the $G$-action on $(Y,\eta)$ a \textbf{continuous Radon model} for \(\alpha\).
We can now state and prove our version of the Glasner-Tsirelson-Weiss result in the context of possibly infinite measures, namely Theorem~\ref{thmi: chara existence spatial model}.

\begin{thm}{\label{thm: chara existence spatial model}}
Let $G$ be a Polish group, and let $\alpha$ be a boolean measure-preserving $G$-action on a standard $\sigma$-finite space $(X,\lambda)$. The following are equivalent:
\begin{enumerate}[(1)]
    \item \label{item: dense G continuous L1}the algebra $\mathcal{G}$ of $G$-continuous functions satisfies  
    \[
    \overline{\mathcal G\cap \LL^1(X,\lambda)}^{\norm\cdot_1}=\LL^1(X,\lambda),
    \]
    \item \label{item: dense G continuous}the algebra $\mathcal{G}$ of $G$-continuous functions satisfies  
\[
\overline{\mathcal G\cap \LL^2(X,\lambda)}^{\norm\cdot_2}=\LL^2(X,\lambda),
\]
\item \label{item: admit continuous model}the action $\alpha$ admits a continuous Radon model.
\end{enumerate}
\end{thm}

\begin{proof}
    In the whole argument, we systematically view boolean $G$-actions as continuous actions by integral preserving $*$-automorphisms on $\LL^\infty$, as per Proposition \ref{prop: boolean action lift} and \ref{prop: aut as integral preserving auto}. We also recall Proposition \ref{prop: chara boolean iso of action}, which allows us to understand boolean isomorphism of action at the level of $\LL^\infty$.\\

    By Proposition \ref{prop: density for subcstar}, \eqref{item: dense G continuous} and \eqref{item: dense G continuous L1} are equivalent, since Corollary \ref{cor: curly G is a Cstar algebra} ensures that $\mathcal{G}$ is a unital $C^*$-subalgebra of $\LL^{\infty}(X,\lambda)$.\\
    
    The implication \eqref{item: admit continuous model}$\Rightarrow$\eqref{item: dense G continuous} is a consequence of the fact that if $\beta$ is the action on the continuous Radon model $(Y,\eta)$ as in \eqref{item: admit continuous model}, then viewed as an action on $\LL^\infty(Y,\eta)$, the restriction of $\beta$ to $C_0(Y)$ is continuous. Since $C_0(Y)$ satisfies $\overline{C_0(Y)\cap \LL^2(Y,\eta)}^{\norm\cdot_2} \subseteq \LL^\infty(Y,\eta)$, we obtain through the boolean isomorphism between $\alpha$ and $\beta$ that $\mathcal G$ satisfies the desired density condition $\overline{\mathcal G\cap \LL^2(X,\lambda)}^{\norm\cdot_2}=\LL^2(X,\lambda)$. \\

    The converse \eqref{item: dense G continuous}$\Rightarrow$\eqref{item: admit continuous model} requires more work, as in \cite{glasnerAutomorphismGroupGaussian2005}. We fix a boolean measure-preserving $G$-action $\alpha$, viewed as a continuous $G$-action by $*$-automorphisms on $\LL^\infty(X,\lambda)$ which preserves the integral. We assume that its space $\mathcal G$ of $G$-continuous functions (which is a $C^*$-algebra by Corollary \ref{cor: curly G is a Cstar algebra}) satisfies condition \eqref{item: dense G continuous}. 

    \paragraph{Step 1. Choosing a suitable separable $\alpha$-invariant $C^*$-subalgebra $\mathcal A\subseteq\mathcal G$.}

    By Proposition \ref{prop: density for subcstar}, the set of finite measure-supported elements of $\mathcal G$ is $\norm{\cdot}_2$-dense in $\mathcal G$.
    Since $\LL^2(X,\lambda)$ is separable for the $\LL^2$ norm, we may
    fix a countable  subset $\mathcal D \subseteq\mathcal G$ consisting of functions whose supports have finite measure, such that $\mathcal D$ is $\norm{\cdot}_2$-dense in $\LL^2(X,\lambda)$.

    Let $\mathcal E=\alpha(G)\mathcal D$, then we claim that $\mathcal E$ is $\norm{\cdot}_\infty$-separable: if $\Gamma$ is a countable dense subset of $G$ then by $G$-continuity the countable set $\alpha(\Gamma)\mathcal D$ is dense in $\mathcal E$.
    If we finally let $\mathcal A$ denote the $C^*$-algebra generated by $\mathcal E$, then $\mathcal A$ is still separable: the countable set of finite $\Q[i]$-linear combinations of finite products of elements of $\alpha(\Gamma)\mathcal D\cup (\alpha(\Gamma)\mathcal D)^*$ is dense therein. 
    By construction $\mathcal A\subseteq\mathcal G$ and $\mathcal A$ is $\norm{\cdot}_\infty$-separable. Moreover, it contains a $\norm{\cdot}_2$-dense subset $\mathcal D$ of $\LL^2(X,\lambda)$ consisting of functions whose supports have finite measure. 

    \paragraph{Step 2. Building the space and the measure.}
    By the theorem of Gelfand-Naimark (Theorem \ref{thm: Gelfand}), the map
    \[
    \begin{array}{ccccc}
    \rho & : & \mathcal{A} & \longrightarrow & C_{0}(\spec(\mathcal A))
    \end{array}
    \]
    given by $\rho(a)(\chi) = \chi(a)$ for any $\chi$ in $\spec(\mathcal{A})$, is an isometric $\ast$-isomorphism. Let $Y\coloneqq \spec \mathcal A$, then since $\mathcal A$ is separable we obtain from Proposition \ref{prop: separable implies second-countable spectrum} that $Y$ is locally compact Polish.

    Now let
    $$\mathcal{I} \coloneqq \left\{ a \in \mathcal{A} \colon \lambda(\supp{a}) < + \infty \right\}.$$ 
    Then $\mathcal I$ is and ideal of $\mathcal{A}$ which is $\norm\cdot_2$-dense in $\LL^2(X,\lambda)$ since it contains $\mathcal D$.

    Functions in $\mathcal{I}$ are essentially bounded and have supports of finite measure, therefore they are integrable, so the measure $\lambda$ defines a positive linear functional 
    \[
    \begin{array}{ccccc}
    \Psi_\lambda & : & \mathcal{I} & \longrightarrow & \C\\
    & & f & \longmapsto & \int_X fd\lambda.
    \end{array}
    \]
    The linear functional $\Psi_\lambda$ can be transported through the Gelfand isomorphism, yielding the positive linear functional $\Psi_{\lambda}\circ\rho\inv : \rho(\mathcal I)\to\C$. 
    Proposition \ref{prop:FaitGeorges} ensures that $C_c(Y) \subseteq \rho(\mathcal{I})$, where $C_c(Y)$ denotes the space of compactly supported continuous functions on $Y$.
    We then restrict $\Psi_{\lambda}\circ\rho\inv$ to $C_c(Y)$, and by Theorem \ref{thm: Riesz} we obtain a unique Radon measure $\eta$ on $Y$
    such that for all $f\in C_c(Y)$,
    \[
    \int_Y fd\eta=\int_X \rho\inv(f)d\lambda.
    \]
    
    \paragraph{Step 3. Extending $\rho$ to the whole $\LL^{\infty}(X,\lambda)$.}

    We will now extend the definition of $\rho$ to functions in $\LL^{\infty}(X,\lambda)$. 
    We are essentially reformulating the uniqueness of the GNS construction
    for weights
    in our restricted setup.
    While the construction is natural, it is quite long to set up and we thus 
    encourage the reader to take this step for granted at first reading.
    The following diagram summarises the situation and the notations that we will use throughout this step.

    \[
    \begin{tikzcd}
	    \LL^{2}(X,\lambda)   \arrow[r, "\displaystyle{\widetilde{\rho}}"] &  \LL^{2}(Y,\eta)    \\
        \mathcal{I} \arrow[hookrightarrow]{d}{\norm{\cdot}_{\infty}\textnormal{-cls}} \arrow[r, "\displaystyle{\rho_{| \mathcal{I}}}"] \arrow[hookrightarrow, swap]{u}{\norm{\cdot}_{2}\textnormal{-cls}}  & \rho(\mathcal{I}) \arrow[hookrightarrow, swap]{d}{\norm{\cdot}_{\infty}\textnormal{-cls}} \arrow[hookrightarrow]{u}{{\norm\cdot}_{2}\textnormal{-cls}}  \\
        \mathcal{A} \arrow[hookrightarrow]{d}{\textnormal{str-cls}} \arrow[loop left, distance=2.7em, start anchor={[yshift=-1ex]west}, end anchor={[yshift=1ex]west},"\alpha"] \arrow[r, "\displaystyle{\rho}"] & C_{0}(Y)  \arrow[hookrightarrow, swap]{d}{\textnormal{str-cls}} \arrow[loop right, distance=2.5em, start anchor={[yshift=1ex]east}, end anchor={[yshift=-1ex]east},"\beta"]{}  \\
        \LL^{\infty}(X,\lambda)  \arrow[hookrightarrow]{d}{M} \arrow[r, "\displaystyle{\overline{\rho}}"] & \LL^{\infty}(Y,\eta)  \arrow[hookrightarrow, swap]{d}{M}  \\
        \mathcal{B}(\LL^{2}(X,\lambda))  \arrow{r}{\displaystyle{\widehat{\rho}}}  &  \mathcal{B}(\LL^{2}(Y,\eta)) 
    \end{tikzcd}
    \]

    \vspace{0.3cm}

     By definition of $\eta$, for any $f$ in $\mathcal{I}$ we have $\int_X fd\lambda = \int_Y \rho(f)d\eta$. In other words, $\rho_{\restriction \mathcal{I}}$ preserves integrals, hence it takes the inner product of $\LL^{2}(X,\lambda)$ to that of $\LL^{2}(Y,\eta)$. Therefore $\rho$ induces a surjective isometry $\widetilde{\rho}$ between the $\norm{\cdot}_2$-closures of $\mathcal{I}$ and $\rho(\mathcal{I})$, which are $\LL^{2}(X,\lambda)$ and $\LL^{2}(Y,\eta)$ respectively, by density of $\mathcal I$ (this is where the fact that $\mathcal D\subseteq\mathcal I$ is crucially used) and by Lemma \ref{lem: dense C0 in L2}. 
     
     We then have a natural $*$-isomorphism $\widehat{\rho}:\mathcal{B}(\LL^{2}(X,\lambda))\to \mathcal B(\LL^2(Y,\eta))$ given by the conjugacy by $\widetilde \rho$, namely for all $f\in \mathcal B(\LL^2(X,\lambda))$
    \[
    \widehat{\rho}(f) = \widetilde{\rho} f\widetilde{\rho}\inv.
    \]
    Being the conjugation by a surjective isometry, $\widehat{\rho}$ takes the strong topology on $\mathcal{B}(\LL^{2}(X,\lambda))$ to the strong topology on $\mathcal B(\LL^2(Y,\eta))$.
    
    Recall that we defined $M$ as the multiplication embedding $f \in \LL^{\infty}(X,\lambda) \mapsto   M_f \in\mathcal{B}(\LL^{2}(X,\lambda))$. We use the same notation for the multiplication embedding $M:\LL^{\infty}(Y,\eta) \rightarrow   \mathcal{B}(\LL^{2}(Y,\eta))$. For any functions $a$ in $\mathcal{A}$ and $\xi$ in $\mathcal{I}$, we have
    \[
    \rho(M_a\xi) = \rho(a \xi) = \rho(a) \rho(\xi) = M_{\rho(a)}\rho(\xi).
    \]
    By density of $\mathcal{I}$ and by continuity of the multiplication, this still holds if we take $\xi$ to be in $\LL^{2}(X,\lambda)$, \textit{i.e.} for any $a$ in $\mathcal{A}$ and any $\xi$ in $\LL^{2}(X,\lambda)$, we have 
    \[
    \widetilde{\rho}(M_a \xi) = M_{\rho(a)} \widetilde{\rho}(\xi).
    \]
    This is equivalent to saying that $\widetilde{\rho} M_a \widetilde{\rho}^{-1} = M_{\rho(a)}$, for any $a$ in $\mathcal{A}$, which can be reformulated as $\widehat{\rho}(M_a) = M_{\rho(a)}$ by definition of $\widehat{\rho}$. Since $\mathcal A\cap \LL^2(X,\lambda)$ is dense in $\mathcal \LL^2(X,\lambda)$, by Proposition \ref{prop: density for subcstar}
    the strong-closure of $M(\mathcal{A})$ in $\mathcal B(\LL^2(X,\lambda))$ is equal to $M(\LL^{\infty}(X,\lambda))$. Moreover, by Lemma \ref{lem: dense C0 in L2} the strong closure of $M(\rho(\mathcal A))=M(C_0(Y))$ in $\mathcal B(\LL^2(Y,\eta))$ is equal to $M(\LL^{\infty}(Y,\eta))$.

    Therefore the isomorphism $\widehat{\rho}$ restricts to an isomorphism of von Neumann algebras between $M(\LL^{\infty}(X,\lambda))$ and $M(\LL^{\infty}(Y,\eta))$ sending the measure $\lambda$ to $\eta$. As the homomorphisms $M$ are isomorphisms on their images, defining $\overline{\rho}$ by 
    \[
    \overline{\rho}(f) = M^{-1}\widehat{\rho}M(f)
    \]
    for any $f$ in $\LL^{\infty}(X,\lambda)$ allows us to extend $\rho$ to $\overline \rho: \LL^{\infty}(X,\lambda)\to \LL^{\infty}(Y,\eta)$.

    \paragraph{Step 4: Defining $\beta$ and showing that it is the action we want.}
    We now show that $G$ acts on $(Y,\eta)$ in a spatial manner, and that this action is booleanly isomorphic to the original boolean $G$-action on $(X,\lambda)$. 

    Since $\overline{\rho}$ is an isomorphism $\LL^\infty(X,\lambda)\to \LL^\infty(Y,\eta)$, we already have our candidate boolean action $\beta$ on $\LL^\infty(Y,\eta)$, uniquely defined by letting, for all $f\in\LL^\infty(X,\lambda)$,
    \[
    \beta(g,\overline{\rho}(f))=\overline{\rho}(\alpha(g,f)).
    \]
    Now $G$ acts on $\mathcal A$ via $\alpha$, yielding a spatial action $\hat\alpha$ on $Y=\spec \mathcal A$ defined by: for any $\chi\in Y$ and any $a\in\mathcal{A}$,
    $\hat\alpha(g, \chi)(a)=\chi( \alpha(g\inv, a))$.
    
    Since $\overline{\rho}$ extends $\rho$, we have in particular that for every $a\in\mathcal A$, 
    $\beta(g,\rho(a))=\rho (\alpha(g,a))$.
    So $\beta(g,\rho(a))(\chi)=\chi( \alpha(g,a))$ for any $\chi$ in $Y$. We can rewrite this as: 
    \begin{equation*}\label{eq: alpha is as wanted on A}
       \beta(g,\rho(a))(\chi)=\chi(\alpha(g,a))=\hat \alpha( g\inv, \chi)(a)  = \rho(a)(\hat \alpha( g\inv, \chi))
    \end{equation*}
    In other words, when restricted to $\rho(\mathcal A)=C_0(Y)$, the action $\beta$ coincides with the precomposition by the inverse associated to the spatial action $\hat \alpha$.
    By strong density of $M(C_0(Y))$ in $M(\LL^\infty(Y,\eta))$ and strong continuity of $\overline{\rho}$, the same conclusion holds on $\LL^\infty(Y,\eta)$, namely $\beta$ coincides with the precomposition by the inverse associated to the spatial action $\hat \alpha$. 
    
    In other words  $\beta$ is the boolean action associated to $\hat \alpha$, and since $\beta$ was integral preserving we have that $\hat\alpha$ is measure-preserving (this could also be checked directly using the Riesz-Markov-Kakutani theorem).
    
    We finally need to check is that $\hat\alpha$ is continuous, using the fact that  $\mathcal{A}$ consists solely of $G$-continuous functions. The topology on $\spec(\mathcal{A})$ is the topology of pointwise convergence, so we fix $a$ in $\mathcal{A}$.
    Let $(g_n)$ be a sequence of elements of $G$ converging to $e_G$ and 
    let $(\chi_n)$ be converging to $\chi$ in $Y=\spec \mathcal A$. We have 
    \begin{align*}
    \abs{\hat\alpha(g_n, \chi_n)(a) - \chi(a)} & = \abs{ \widehat{a}(\hat\alpha(g_n, \chi_n))- \chi(a) }\\
    & \leqslant \abs{ \widehat{a}(\hat\alpha(g_n, \chi_n)) - \widehat{a}(\chi_n) } + \abs{   \widehat{a}(\chi_n) - \widehat{a}(\chi)           }\\
    & \leqslant \norm{ \widehat{\alpha(g_n\inv,a)} - \widehat{a} }_{C_0(Y)} + \abs{ \chi_n(a) - \chi(a) }\\
    & \leqslant \norm{ \alpha(g_n\inv, a) - a }_{\infty} + \abs{ \chi_n(a) - \chi(a) }.
    \end{align*}
    By convergence of $(\chi_n)$ the second term tends to zero as $n$ grows, and the first term tends to zero as $n$ grows by $G$-continuity of $a$.
    Therefore the $G$-action on $Y$ is continuous as wanted. 
\end{proof}

\section{Applications and examples}

\subsection{Construction of \texorpdfstring{$G$}{G}-continuous functions when \texorpdfstring{$G$}{G} is locally compact}
\label{sec: convolution}
In the whole section, we work with a fixed locally compact Polish group $G$, and denote by $m$ its Haar measure.
We start by recalling the definition of convolution between an essentially bounded function on a measure space and a function that is integrable with regards to the Haar measure of the group acting on said space. 

	\begin{defi}
Consider a boolean measure-preserving $G$-action $\alpha$ on $(X,\lambda)$ as in Definition \ref{def: boolean mp or ns action}. Let $f\in \LL^\infty(X,\lambda)$ and $\delta\in\LL^{1}(G,m)$.
For any $x$ in $X$, the \textbf{convolution product} $\delta*f: X\to \C$ is defined by: for all $x\in X,$
\[
\delta * f(x) = \int_G \delta(g)f(\alpha(g^{-1},x))dm(g).
\]
	\end{defi}

Using for $\delta$ an approximation of the identity, we will see that  $G$-continuous functions are always dense for boolean measure-preserving actions. We start by proving the following well-known lemma.

\begin{lem}
\label{lem:convoGcont}
Consider a locally compact Polish group $G$, and a boolean measure-preserving action $\alpha$ of $G$ on a standard $\sigma$-finite space $(X,\lambda)$.
Let $\delta : G \rightarrow \mathbb{R}_+$ be continuous of integral $1$ and compactly supported. Then, for any function $f$ in $\LL^{\infty}(X,\lambda) \cap \LL^{1}(X,\lambda)$, $\delta * f$ is in $\LL^{\infty}(X,\lambda) \cap \LL^{1}(X,\lambda)$ with $\norm{\delta*f}_1\leq\norm{f}_1$, and is $G$-continuous.
\end{lem}

	\begin{proof}
We first check that $\delta * f$ is in $\LL^1(X,\lambda)$, as it is clear that it is in $\LL^\infty(X,\lambda)$. Using Fubini's theorem and the fact that $\delta$ takes nonnegative values we have
\begin{align*}
\int_X \abs{ \delta * f(x) } d\lambda(x) 
& = \int_X \abs{ \int_G \delta(g)f(\alpha(g\inv,x))dm(g) } d\lambda(x)\\
& \leqslant \int_X \int_G \abs{ \delta(g) f(\alpha(g\inv,x)) } dm(g) d\lambda(x) \\
& = \int_G \delta(g) \int_X  \abs{ f(\alpha(g\inv,x)) } d\lambda(x) dm(g)\\
& = \norm{f}_1,
\end{align*}
the last equality being a consequence of the fact that $\int_G \delta(g)dm(g)=1$ and $g$ preserves the measure. In particular we have $\delta* f\in \LL^1(X,\lambda)$ and moreover $\norm{\delta*f}_1\leq \norm{f}_1$.


We finally prove that it is $G$-continuous. Let us take $\varepsilon>0$.  For any $h$ in $G$ we have
\[
\delta * f(\alpha(h,x)) = \int_G \delta(g)f(\alpha(g\inv h,x))dm(g) = \int_G \delta(hg)f(\alpha(g\inv,x))dm(g).
\]
Therefore, we have
\begin{equation}\label{eq: towards G continuity}
\abs{\delta * f(x) - \delta * f(\alpha(h,x))} \leqslant \int_G \abs{ \delta(g) - \delta(hg) }\abs{ f(\alpha(g\inv,x)) } dm(g).
\end{equation}
By Lemma \ref{lem: G continuous from continuous action on lc} for the $G$-action on itself by left translation, we can fix a neighborhood $\mathcal{N}$ of $e_G$ such that for all $h\in \mathcal{N}$ and $g\in G$, $\abs{\delta(g)-\delta(hg)}<\varepsilon$. 
Also, note that if $K=\supp \delta$, then $\delta(g)-\delta(hg)=0$ for all $g\notin K\cup h\inv K$, which has measure at most $2m(K)$ since $m$ is a left Haar measure.
It follows that
\begin{align*}
\int_G \abs{\delta(g) - \delta(hg)} dm(g) &< 2\varepsilon \times m(K)
\end{align*}
Therefore we can control the left hand side of Equation \eqref{eq: towards G continuity} by the $\LL^{\infty}$ norm of $f$:
\[
\abs{\delta * f(x) - \delta * f(\alpha(h,x))} < 2 \varepsilon \times m(K) \times \norm{f}_{\infty}
\]
which concludes the proof.
	\end{proof}

We can now prove Theorem~\ref{thmi: iso lc spatial model}
in the special case that the acting group is locally compact.

\begin{thm}{\label{thm: continuous radon for lcsc}}
    Let $G$ be a locally compact Polish group, and consider a boolean measure-preserving $G$-action on a standard $\sigma$-finite space $(X,\lambda)$. Then the action admits a continuous spatial model on a locally compact Polish space $(Y,\eta)$, where $\eta$ is a Radon measure.
\end{thm}

\begin{proof}
    We want to apply condition \eqref{item: dense G continuous L1} from Theorem \ref{thm: chara existence spatial model} in order to conclude. 
    
    Since bounded functions are dense in $\LL^1(X,\lambda)$, 
    we fix some $f\in\LL^\infty(X,\lambda)\cap\LL^1(X,\lambda)$, some $\varepsilon>0$, and our aim is to find a $G$-continuous function which is $\varepsilon$-close to $f$ for $\norm{\cdot}_1$.
     
    For any $\delta\in\LL^{1}(G,m)$ compactly supported of integral $1$ taking nonnegative values, we know from Lemma \ref{lem:convoGcont} that $\delta * f$ is in $\LL^{\infty}(X,\lambda) \cap \LL^{1}(X,\lambda)$ and is $G$-continuous. 
    By Fubini's Theorem, we also have

    \begin{align*}
        \| f - \delta * f \|_{1} & = \int_X \abs{f(x) - \int_G \delta(g)f(\alpha(g\inv,x))dm(g) } d\lambda(x)\\
        & = \int_X \abs{ \int_G \delta(g)f(x)dm(g) - \int_G \delta(g)f(\alpha(g\inv,x))dm(g) } d\lambda(x)\\
        & \leqslant \int_X \int_G \delta(g) \abs{f(x) - f(\alpha(g\inv,x)) } dm(g)d\lambda(x) \\
        & = \int_G \delta(g) \int_X \abs{f(x) - f(\alpha(g\inv,x))} d\lambda(x) dm(g)\\
        & = \int_G \delta(g) \| f - f \circ \alpha(g\inv,\cdot) \|_{1} dm(g)
    \end{align*}
    
    But as noted in Section \ref{sec: actions on function spaces}, the $G$-action on $\LL^{1}(X,\lambda)$ is continuous. Therefore there exists a neighborhood $\mathcal{N}_\varepsilon$ of $e_G$ such that $\| f - f \circ \alpha(g\inv,\cdot) \|_{1} < \varepsilon$ for any $g$ in $\mathcal{N}_\varepsilon$. Taking now $\delta$ as above whose support is moreover contained in $\mathcal{N}_\varepsilon$, we obtain $\norm{f-\delta*f}_1 < \varepsilon$, which concludes the proof.
\end{proof}

Combining Mackey's theorem (Theorem \ref{thm: mackey}), and Theorem \ref{thm: continuous radon for lcsc}, we obtain Theorem \ref{thm intro: spatial conjugacy to nice action}
as a corollary.
\begin{cor}{\label{cor: radon model for lcsc via Mackey}}
    Let $G$ be a locally compact Polish group and $(X,\lambda)$ be a standard $\sigma$-finite space. Let also $\alpha$ be a measure-preserving $G$-action on $(X,\lambda)$. Then $\alpha$ is spatially isomorphic to a continuous measure-preserving $G$-action on a locally compact Polish space $Y$ endowed with a Radon measure $\eta$.
\end{cor}
\begin{proof}
    By Theorem \ref{thm: continuous radon for lcsc} the boolean action associated with $\alpha$ admits a Polish model $\beta$ on $(Y,\eta)$, where $Y$ is locally compact and Polish, and $\eta$ is a Radon measure. By Mackey's theorem (item \eqref{item:Mackey2} of Theorem \ref{thm: mackey}), since $\alpha$ and $\beta$ are booleanly isomorphic, they are in fact spatially isomorphic.
\end{proof}

In section \ref{sec: spatial Radon model lcsc} we give another way to obtain Corollary \ref{cor: radon model for lcsc via Mackey} without relying on the use of Mackey's theorem. By working directly with measurable bounded functions and the genuine supremum norm instead of equivalence classes of functions up to measure zero, we explicitely construct the desired spatial isomorphism.

\subsection{Actions of isometry groups of locally compact separable spaces}

In this section, building upon Theorem \ref{thm: continuous radon for lcsc} and Theorem
\ref{thm: chara existence spatial model},
we extend the result of Kwiatkowska-Solecki 
on the existence of continuous spatial models 
for isometry groups of locally compact separable spaces to our $\sigma$-finite setup.
As in their paper, we rely on the following characterization of this class of 
isometry groups (endowed with the topology of pointwise convergence). 

\begin{thm}[{Kwiatkowska-Solecki, see \cite[Thm.~1.2]{kwiatkowskaSpatialModelsBoolean2011}}]
Let $G$ be a Polish group. Then $G$ is topologically isomorphic 
to the group of isometries of
a locally compact separable metric space if and only 
if it satisfies the following property:
\begin{enumerate}
    \item[($*$)]\phantomsection
    \label{eq: chara iso lcsc}
for every open neighborhood of the identity $U \subseteq G$ there is a closed
subgroup $H < G$ such that \(H \subseteq U\), \(N (H)\) is open, and 
the Polish space \(G/H\) is
locally compact.
\end{enumerate}
\end{thm}

Let us remark that in \hyperref[eq: chara iso lcsc]{($*$)}, the 
normalizer subgroup $N(H)$ is open, hence closed, in particular 
$N(H)/H$ is a closed subset of the locally compact space $G/H$.
Since it is moreover a topological group, we conclude that 
$N(H)/H$ is a locally compact Polish group.

\begin{prop}\label{prop: dense Gcont for iso of lc}
    Let $G$ be the isometry
    group of a locally compact separable metric space, 
    and consider $\alpha$ a boolean measure-preserving $G$-action 
    on a standard $\sigma$-finite space $(X,\lambda)$. 
    The set of square integrable $G$-continuous functions with regards to this action 
    is dense in $\LL^2(X,\lambda)$.
\end{prop}
\begin{proof}
    Consider an element $f$ of $\LL^2(X,\lambda)$, and fix $\varepsilon > 0$. 
    As $G$ acts continuously on $\LL^2(X,\lambda)$ 
    there exists an open  $U$ such that for any $h$ in $U$ we have 
    $\norm{f - f \circ \alpha(h\inv, \cdot)}_{2} <  \varepsilon$. 
    As $G$ satisfies 
    Condition \hyperref[eq: chara iso lcsc]{($*$)}, 
    we can fix a closed subgroup $H < G$ such that
    $H \subseteq U$, the quotient space $G/H$ is locally compact and the normalizer
    $N(H)$ is open.
    
    Consider now 
    \[
    \mathcal{C} = \overline{\mbox{Conv}\{(f \circ \alpha(h\inv, \cdot)) \mid h \in H\}}^{\norm\cdot_2},
    \]
    the $\LL^2$-closure of the convex hull of the orbit of $f$ under the $H$-action. 
    Note that $\mathcal{C}$  is contained it the closed ball centered in $f$ 
    and of radius $\varepsilon$.
    
    By the projection onto a closed convex set theorem 
    (see e.g.~\cite[Thm.~4.10.]{rudinRealComplexAnalysis1987}), 
    there is a unique element of minimal norm $\widehat{\xi}$ in $\mathcal{C}$, 
    satisfying in particular $\norm{f - \widehat{\xi}}_2 \leq \varepsilon$.
    Since the group $H$ acts by unitaries on $\mathcal{C}$, the vector
    $\widehat{\xi}$ is fixed by the $H$-action. 
    In particular, $\widehat{\xi}$ is $H$-continuous, although it might be unbounded
    and have support of infinite measure. 
    To conclude the first part of the proof 
    we approximate $\widehat{\xi}$ by functions in $\LL^{\infty}(X,\lambda) \cap \LL^{2}(X,\lambda)$
    whose support has finite measure 
    while retaining the $H$-invariance. 
    Let \[
    \xi_n \coloneqq \mathds{1}_{\left\{ \abs{\widehat{\xi}} \leqslant n \right\}}\mathds{1}_{\left\{ \abs{\widehat{\xi}} \geqslant \frac 1n \right\}} \widehat{\xi}.
    \] 
    For any $n$ the function $\xi_n$ is in 
    $\LL^{\infty}(X,\lambda) \cap \LL^{2}(X,\lambda)$, has support of finite measure, 
    and is $H$-invariant.
    As $\xi_n \to \widehat{\xi}$ in $\norm{\cdot}_{2}$, we can fix, 
    for some $n$ large enough, a function $\xi\coloneqq \xi_n$ that is $\varepsilon$-close to $\widehat{\xi}$,
    and hence $2\varepsilon$-close to $f$, while being $H$-invariant and having
    a support of finite measure.\\

    Let now $\mathcal M$ be the von Neumann algebra generated by the $N(H)$-translates
    of $\xi$, i.e. the strong closure in $\LL^\infty(X,\lambda)$
    of the linear span of the $N(H)$-translates of 
    $\xi$ and of the constant function $1$. We claim that any element of $\mathcal M$ is fixed by $H$. Indeed, any $N(H)$-translates of $\xi$ is easily checked to be $H$-invariant, 
    so the linear span of the $N(H)$-translates of $\xi$ consists
    of $H$-invariant vectors. Finally, since each $h\in H$ acts 
    continuously on \(\LL^\infty(Y,\lambda)\) for the strong operator topology, its set of 
    fixed points is closed, so \(\mathcal M\) consists of $H$-invariant functions.

    We now argue, using a simpler version of the arguments
    in the proof of Theorem \ref{thm: chara existence spatial model}, that
    the $N(H)$-action on $(\mathcal M,\lambda)$ can be seen as a boolean action on a 
    $\sigma$-finite 
    space, possibly with atoms. It will then be easy to see that 
    it descends to a boolean action of the locally compact Polish group 
    $N(H)/H$, allowing us to use Theorem \ref{thm: continuous radon for lcsc}.\\

    Denote by $\mathcal J$ the ideal in $\LL^\infty(X,\lambda)$ of functions 
    whose support has finite measure.
    First observe that since $\xi$ belongs to $\mathcal J$ and the $N(H)$-action preserves
    the measure $\lambda$,
    we have that $\mathcal J\cap \mathcal M$ is strongly dense in $\mathcal M$.
    In a similar fashion to the first step of the proof of Theorem \ref{thm: chara existence spatial model}, we let $\mathcal D$ be a countable dense subset of 
    $\mathcal J\cap \mathcal M\cap \LL^2(X,\lambda)$ for the $\LL^2$ norm, and let $\mathcal A$ be  
    the $\norm\cdot_\infty$-closure of the linear span of 
    $\mathcal D$.
    Then $\mathcal A$ is a separable $C^*$-subalgebra 
    of $\mathcal M$ which is strongly dense in $\mathcal M$, and such that the ideal
    $\mathcal I \coloneqq \mathcal A\cap \mathcal J$ is $\norm{\cdot}_\infty$-dense 
    in $\mathcal A$.
    Using now the Gelfand theorem exactly as in the second step of the proof of
    Theorem \ref{thm: chara existence spatial model}, we have a locally 
    compact Polish space $Y$ endowed with a Radon measure $\eta$
    so that after identification through the Gelfand isomorphism, 
    $\mathcal A=C_0(Y)$ and hence by strong density $\mathcal M=\LL^\infty(Y,\eta)$.
    By construction $\eta$ and $\lambda$ coincide on $\mathcal I$,
    in particular they define the same measure on $Y$.
    So the $N(H)$-action on $\mathcal M=\LL^\infty(Y,\eta)$ is continuous 
    in the sense that we gave right before Proposition \ref{prop: chara boolean iso of action},
    namely for every $A\in\MAlg_f(Y,\eta)$, and every $g_n\to e_G$ in $N(H)$, 
    \[
    \int_Y\abs{g_n\cdot \chi_A- \chi_A} d\eta \to 0. 
    \]
    
    We thus now have a Boolean action of $N(H)$ on $(Y,\eta)$, where $(Y,\eta)$ 
    is a standard $\sigma$-finite space, possibly with atoms, 
    and moreover $\mathcal M=\LL^\infty(Y,\eta)$.
    But since every element of $\mathcal M$ is fixed by $H$, the restriction
    of this boolean action to $H$ is trivial and the boolean action
    thus descends to a boolean action of the quotient group $N(H)/H$, which is a locally compact Polish group. 
    Applying Theorem \ref{thm: continuous radon for lcsc} along with Theorem \ref{thm: chara existence spatial model}, we now find 
    $\xi'\in\LL^2(Y,\eta)\subseteq\LL^2(X,\lambda)$ which is 
    $\varepsilon$-close to $\xi$ and $N(H)/H$-continuous, and hence
    $N(H)$-continuous when we view again this $N(H)/H$-action as an $N(H)$-action.
    
    Since $N(H)$ is open in $G$, we have in particular that $\xi'$,
    seen as an element of $\LL^2(X,\lambda)$, 
    is $G$-continuous, and by the triangle inequality it is $3\varepsilon$-close
    to the function $f$ we started with, which finishes the proof.
\end{proof}

We can now directly prove Theorem~\ref{thmi: iso lc spatial model}.

\begin{thm}\label{thm: iso lc spatial model}
Let $G$ be the isometry group of a separable locally compact space.
Then every  boolean measure-preserving $G$-action 
on a standard $\sigma$-finite space $(X,\lambda)$ 
admits a continuous Radon model: 
\end{thm}
\begin{proof}
By Proposition \ref{prop: dense Gcont for iso of lc}, 
the $G$-continuous functions are dense in $\LL^{2}(X,\lambda)$, 
and we can then apply Theorem \ref{thm: chara existence spatial model} 
to conclude the proof.
\end{proof}

\begin{rem}
By the above theorem, it is always possible to find a continuous Radon model for 
any boolean measure-preserving action of 
the non-archimedean Polish group $\Sinf$. 
However, as we explain in Section~\ref{sec: non radonable actions}, 
it is not always possible to turn  boolean isomorphisms between $\Sinf$ 
spatial actions
into  spatial isomorphisms, as opposed to
what happens locally compact groups (see the
second item of Theorem~\ref{thm: mackey}). 
\end{rem}
\begin{rem}
    It is a result of Nessonov that for every boolean \emph{non-singular} $\Sinf$-action $\alpha$ on $(X,\lambda)$, there is an equivalent $\sigma$-finite measure $\lambda'\in[\lambda]$ such that $\alpha$ becomes measure-preserving \cite[Thm.~1.1]{nessonovNonsingularAction2020}. So our result yields that every boolean  non-singular $\mathfrak S_\infty$-action admits a continuous Radon model.
\end{rem}

\subsection{An ergodic boolean action with two distincts spatial models}{\label{sec: two realizations}}

In this section we present an example of an ergodic boolean action with two distincts spatial models on probability spaces, which was explained to us by Todor Tsankov. The action is that of $\Sinf$ on the countable product of the unit interval endowed with the Lebesgue measure $\left( \left[ 0 , 1 \right]^{\N}, \lambda^{\otimes \N} \right)$. The first spatial action $\alpha$ is the natural action of $\Sinf$ by permutation of the coordinates. For the second action $\beta$, we first define the space $\LON$ of linear orderings of $\N$:
\[
\LON = \left\{ x \subseteq \N^2 \mid x \mbox{ is a linear order on } \N \right\},
\]
and for any element $x$ (that we will denote in a more usual way by $n <_x m \Longleftrightarrow (n,m) \in x$) of $\LON$, the $\Sinf$-action $\beta$ is defined by
\[
(n,m) \in  \beta(\sigma,x)  \Longleftrightarrow n (\sigma \cdot <_x ) m \Longleftrightarrow \sigma^{-1}(n) <_x \sigma^{-1}(m)
\]
for any $\sigma$ in $\Sinf$. We define on $\LON$ the uniform measure $\mu_u$ by
\[
\mu_u \left( \left[ k_1 <_x k_2 <_x \ldots <_x k_n \right] \right)  = \dfrac{1}{n!}
\]
for any distinct $k_1,\ldots,k_n$ in $\N$, where 
\[
\left[ k_1 <_x k_2 <_x \ldots <_x k_n \right] = \left\{ x \in \LON \mid  k_1 <_x k_2 <_x \ldots <_x k_n  \right\}.
\]
This probability measure is invariant under the $\Sinf$-action $\beta$, and do note that it is the only such measure by construction. Do recall also that, as $\LON$ is a compact space as a closed subset of $\left\{ 0,1 \right\}^{\N^2}$, $\Sinf$ is not extremely amenable. Indeed, $\Sinf$ contains transpositions, and in particular it cannot preserve a linear ordering of $\N$, hence the previous assertion.

The correspondence between the two actions is defined as follows:
\[
\begin{array}{ccccc}
    \Phi & : & \left[ 0 , 1 \right]^{\N} & \longrightarrow & \LON \\
     & & (x_i)_{i \in \N} & \longmapsto & <_{x},
\end{array}
\]
where $<_x$ is defined by $n <_x m \Longleftrightarrow x_n < x_m$ (the order $<$ on the right-hand side being the natural order on $\left[ 0 , 1 \right]$). Note that for $\lambda^{\otimes \N}$-almost all $x$ in $\left[ 0 , 1 \right]^{\N}$, the $x_n$ are distinct, so $\Phi$ is well defined. We then show that the images by $\Phi$ of two distinct elements are isomorphic as orderings of $\N$.
    \begin{defi}
An order $<$ on a set $A$ is \textbf{dense on $A$} if for any two elements $n,m$ in $A$ such that $n<m$, there exists a third element $k$ in $A$ such that $n<k<m$.
    \end{defi}
    \begin{defi}
An order $<$ on a set $A$ is \textbf{unbounded on $A$} if for any element $k$ in $A$, there exists two elements $n$ and $m$ in $A$ such that $n<k<m$.
    \end{defi}
    \begin{prop}
An order $<$ in $\LON$ is dense and unbounded on $\N$, $\mu_u$-a.e.
    \end{prop}
    \begin{proof}
Fix $n,m$ and $k$ in $\N$. We have 
\[
\mathbb{P}_{\mu_u}\left (\left[n<k<m\right] \mid \left[ n<m \right] \right) = \dfrac{1}{3},
\]
therefore by Borel-Cantelli's Lemma, as $n$ and $m$ vary, $n<k<m$ happens infinitely often, and thus $<$ is unbounded $\mu_u$-a.e. The same argument used with $k$ varying instead also ensures that $<$ is dense $\mu_u$-a.e. which concludes the proof.
    \end{proof}
    The following back-and-forth argument is well-known and classical, but we give it for the sake of completeness.
    \begin{prop}{\label{Orderisom}}
Any two dense and unbounded orders on infinite countable sets are isomorphic, \textit{i.e.} there exists an order preserving bijection between the respective spaces.
    \end{prop}
    \begin{proof}
Let $A = \left\{ a_i \mid i\in \N \right\}$ and $B = \left\{ b_i \mid i\in \N \right\}$ be two countable infinite sets, and let $<_A$ and $<_B$ be two dense, unbounded linear orderings on $A$ and $B$ respectively. We inductively construct an embedding $\varphi : (A,<_A) \hookrightarrow (B,<_B)$.\\
We arbitrarily send $a_0$ to $b_0$. We then send $a_1$ to $\varphi(a_1)$ such that $(b_0,\varphi(a_1))$ is ordered like $(a_0,a_1)$. At rank $n$, $a_n$ is sent to $\varphi(a_n)$, an element that respects the order of $(a_0,\ldots,a_n)$, which necessarily exists because $<_B$ is dense and unbounded.

In other words, for any finite subset $F$ of $(A,<_A)$, and for any $a$ in $A\setminus F$, it is possible at any rank to extend the existing embedding $F \hookrightarrow (B,<_B)$ into $F\cup \left\{ a \right\} \hookrightarrow (B,<_B)$.\\
By doing the same construction with $(B,<_B)$, and given that $<_B$ is also dense and unbounded, it is possible to define by induction 
\[
\varphi_n : F_n \subseteq (A,<_A) \hookrightarrow (B,<_B)
\]
such that $F_n$ is a finite subset of $A$, and such that for any $n$ we have
\begin{align*}
\left\lbrace
\begin{array}{l}
\left\{ a_0 , \ldots , a_n \right\} \subseteq F_{2n}\\
\left\{ b_0 , \ldots , b_n \right\} \subseteq \varphi_{2n+1}(F_{2n+1})
\end{array}
\right.
\end{align*}
The limit function of the $\varphi_n$ is the order-preserving bijection we sought.
    \end{proof}
Proposition \ref{Orderisom} ensures that $\mu_u$-a.e. there is only one $\Sinf$-orbit on $\LON$. On the other hand, the natural $\Sinf$-action on  $\left( \left[ 0 , 1 \right]^{\N}, \lambda^{\otimes \N} \right)$ only admits orbits of measure $0$. We then have the following proposition.

\begin{prop}\label{prop: todor example}
    The application $\Phi : \left( \left[ 0 , 1 \right]^{\N}, \lambda^{\otimes \N} \right) \longrightarrow (\LON, \mu_u)$, 
    sending $(x_i)$ to $<_x$, where $<_x$ is defined by $n <_x m \Longleftrightarrow x_n < x_m$, is $\Sinf$-equivariant, measure-preserving and $\lambda^{\otimes \N}$-a.e. injective.
\end{prop}
\begin{proof}
Checking the $\Sinf$-equivariance is straightforward, and so is checking the fact that $\Phi$ is measure-preserving. For the $\lambda^{\otimes \N}$-a.e. injectivity we first fix $q$ in $\Q \cap \left[ 0 , 1 \right]$. The Law of Large Numbers ensures that for $\lambda^{\otimes \N}$-almost any sequence $(x_i)$ we have
\begin{equation}
	\label{eq: cv law large numbers}
	\lim_{n\to+\infty}\frac{ \abs{\left\{ i \in \{ 0,\dots,n-1 \} \colon x_i < q  \right\}} }{n} = q.
\end{equation}

Taking a countable intersection of conull sets, we can therefore find a conull subset $X_0$ of $\left[ 0 , 1 \right]^{\N}$ such that whenever $(x_i)$ is in $X_0$, Equation \eqref{eq: cv law large numbers} holds for any $q$ in $\Q \cap \left[ 0 , 1 \right]$. A straightforward density argument now shows that Equation \eqref{eq: cv law large numbers} holds for any $q$ in $\left[ 0 , 1 \right]$ and any $(x_i)\in X_0$. 

Let us conclude by showing that $\Phi_{\restriction X_0}$ is injective: take distinct elements $x=(x_i)$ and $y=(y_i)$ in $X_0$, there is an index $k$ such that $x_k \neq y_k$. 
Then taking $q = x_k$ and $q = y_k$ for $(x_i)$ and $(y_i)$ respectively in \eqref{eq: cv law large numbers}, we obtain that
\[
\lim_{n\to+\infty}\frac{ \abs{\left\{ i \in \{ 0,\dots,n-1 \} \colon x_i < x_k  \right\}} }{n} \neq \lim_{n\to+\infty}\frac{ \abs{\left\{ i \in \{ 0,\dots,n-1 \} \colon y_i < y_k  \right\}} }{n}.
\]
This implies that $<_x \neq <_y$, as their respective proportions of integers smaller than the integer $k$ are different.
\end{proof}

The previous proposition gives an example of a negative answer to the boolean to spatial isomorphism problem for the case of the non-locally compact group $\Sinf$, for actions on probability spaces:
\begin{prop}
    Let $G=\Sinf$. The $G$-action $\alpha$ on $\left( \left[ 0 , 1 \right]^{\N}, \lambda^{\otimes \N} \right)$ and the $G$-action $\beta$ on $(\LON, \mu_u)$ are booleanly isomorphic, but not spatially isomorphic.
\end{prop}
\begin{proof}
    We define $\Phi$ as in Proposition \ref{prop: todor example}. This application is (well-defined and) injective on a conull subset $X_0 \subseteq \left[ 0 , 1 \right]^{\N}$. As it is also $G$-equivariant, setting $X_g = X_0 \cap g \cdot X_0$ for any $g$ in $G$ defines a boolean isomorphism between $\alpha$ and $\beta$. They are not spatially isomorphic however, since $\alpha$ only has orbits of measure $0$, but Proposition \ref{Orderisom} ensures that $\beta$ has a conull orbit.
\end{proof}

\section{Spatial actions and continuous Radon models}

In this section, we are interested in a related question on Borel $G$-actions on infinite measured spaces:
given such an action, can it be Borel embedded $G$-equivariantly 
into a continuous $G$-action 
on a locally compact Polish space so that the pushforward measure is a Radon measure?

We will first give an abstract characterization of this, and
then show that the answer is always positive for locally compact groups, 
thus proving Theorem \ref{thmi: locally compact spatical into continuous Radon}.
After that, we discuss the related question of continuous realizations of Borel infinite measure-preserving actions. Finally, we
give a simple example showing our embedding theorem fails for the non locally compact group $G=\mathfrak S_\infty$.

\subsection{Characterization}

In this section, we will work in the space $\mathcal L^\infty(X)$ which we define as the $C^*$algebra of all bounded Borel functions $X\to \C$,
endowed with the norm of uniform convergence $\norm{f}_\infty=\sup_{x\in X} \abs{f(x)}$ (no essential supremum here!).
Observe that if $G$ acts on $X$, the notion of $G$-continuous element of $\mathcal L^\infty(X)$ still makes sense for the norm $\norm{\cdot}_\infty$ that we just defined.
We also denote by $\mathcal{L}^{1}(X,\lambda)$ the space of integrable complex-valued functions on $X$.
\begin{thm}{\label{thm: admits model is equivalent to functions separate points}}
    Let $G$ be a Polish group, suppose $\alpha : G\curvearrowright (X,\lambda)$ is a Borel action by measure-preserving bijections on a standard $\sigma$-finite space $(X,\lambda)$. The following are equivalent:
    \begin{enumerate}[(i)]
        \item \label{cond: spatial continuous in lc radon} There is a continuous $G$-action $\beta$ on a Polish locally compact space $Y$ and a Borel injective map $\Phi: X\to Y$
        such that $\Phi_*\lambda$ is Radon and $\Phi(\alpha(g,x))=\beta(g,\Phi(x))$ for all $x\in X$ and all $g\in G$; 
        \item \label{cond: G continuous countable separates} The algebra of $G$-continuous  functions in $\mathcal L^\infty(X)\cap \mathcal L^1(X,\lambda)$ contains a countable subset which separates points.
    \end{enumerate}
\end{thm}

\begin{proof}
    Let us begin by the easier implication \eqref{cond: spatial continuous in lc radon}$\Rightarrow$\eqref{cond: G continuous countable separates}, and assume we have $\alpha:G\times X\to X$ Borel preserving a $\sigma$-finite measure $\lambda$, $\beta: G\times Y\to Y$ continuous and $\Phi: X\to Y$ injective such that for all $x\in X$, $\Phi(\alpha(g,x))=\beta(g,\Phi(x)).$

    By compactness, every compactly supported continuous function on $Y$ is $G$-continuous with respect to the action $\beta$. 
    Since $Y$ is locally compact second-countable, it admits a compatible proper metric $d$. Let $\{y_n\colon n\in\N \}$ be a countable dense subset of $Y$, and define for each $n\in\N$ a continuous compactly supported function $f_{n}$ by
    $$
    f_{n}(x)= \max \left( 1-d(y_n,x),0 \right).
    $$
    The density of $(y_n)$ implies that this family of $G$-continuous functions separates the points of $Y$. 
    It then follows from the injectivity of $\Phi$ that $(f_{n}\circ\Phi)_{n\in\N}$ separates the points of $X$, and by the equivariance of $\Phi$ it consists of $G$-continuous functions with respect to $\alpha$ as wanted.\\

    We now prove the converse implication 
    \eqref{cond: G continuous countable separates}$\Rightarrow$\eqref{cond: spatial continuous in lc radon}.
    Let us fix a countable set of functions $\mathcal{D} \subseteq L^\infty(X)\cap \mathcal L^1(X,\lambda)$ that separates points. We then consider $\Gamma$ a countable dense subgroup of $G$ and give a similar argument to the one in Step 1 of the proof of Theorem \ref{thm: chara existence spatial model}. By $G$-continuity the $\norm{\cdot}_{\infty}$-closure of $\mathcal{E}_\Gamma \coloneqq \left\{ f \circ \alpha(\Gamma, \cdot) \mid f \in \mathcal{D}\right\}$ is equal to $\mathcal{E}_G \coloneqq \left\{ f \circ \alpha(G, \cdot) \mid f \in \mathcal{D} \right\}$. We denote by $\mathcal{A}$ the $C^*$-algebra generated by $\mathcal{E}_G$. The countable set of finite $\Q[i]$-linear combinations of finite products of elements of $\mathcal{E}_\Gamma \cup (\mathcal{E}_\Gamma)^*$ is dense in $\mathcal{A}$. By construction $\mathcal{A}$ is then separable, as the $\norm{\cdot}_{\infty}$-closure of a countable set (of integrable functions). Moreover every function in $\mathcal{A}$ is $G$-continuous and $\mathcal{A}$ separates the points of $X$.

     \paragraph{Step 1. Building the space and the measure.}

    By theorem \ref{thm: Gelfand}, we have the following isometric $*$-isomorphism:
    \[
    \begin{array}{ccccc}
    \rho & : & \mathcal{A} & \longrightarrow & C_{0}(\spec(\mathcal A))\\
    & & a & \longmapsto & \widehat{a}
    \end{array}
    \]
    with $\widehat{a}$ being the evaluation on $a$. By separability of $\mathcal{A}$ and Proposition \ref{prop: separable implies second-countable spectrum}, the space $Y \coloneqq \spec(\mathcal{A})$ is locally compact Polish. 
    
     This time, we consider 
    \[
    \mathcal{I} = \mathcal{A} \cap \mathcal{L}^{1}(X,\lambda).
    \]
    It is immediate to check that $\mathcal I$ is an ideal of $\mathcal{A}$, and it contains $\mathcal{D}$ by construction. Since $\mathcal I$ consists of integrable functions, $\lambda$ defines a positive linear functional
    \[
    \begin{array}{ccccc}
    \Psi_\lambda & : & \mathcal{I} & \longrightarrow & \C\\
    & & f & \longmapsto & \int_X fd\lambda.
    \end{array}
    \]
    We then use $\rho$ to get a positive linear functional $\Psi_\lambda\circ\rho\inv : \rho(\mathcal{I}) \to \C$. The ideal $\mathcal{I}$ is $\norm{\cdot}_{\infty}$-dense in $\mathcal{A}$, so by Proposition \ref{prop:FaitGeorges} $C_c(Y) \subseteq \rho(\mathcal{I})$, where $C_c(Y)$ denotes the space of compactly supported continuous functions on $Y$. By Theorem \ref{thm: Riesz}, restricting $\Psi_{\lambda}\circ\rho\inv$ to $C_c(Y)$ gives a unique Radon measure $\eta$ on $Y$ such that for all $f \in C_c(Y)$, 
    \begin{equation}{\label{eq: equality of integrals for the measure}}
    \int_Y f d\eta = \int_X \rho\inv(f) d\lambda.
    \end{equation}

    \paragraph{Step 2. Borel embedding of $X$ into $Y$.}

    Consider the following map
    \[
    \begin{array}{ccccc}
    \Phi & : & X & \longrightarrow &  Y \\
    & & x & \longmapsto & \mathrm{ev}_x 
    \end{array}
    \]
    where $\mathrm{ev}_x(a) = a(x)$.
    Since the algebra $\mathcal{A}$ separates the points of $X$, this map is injective. Let us prove that it is also Borel. Sets of the form
    \[
    \left\{ \mathrm{ev}_x \colon \abs{\mathrm{ev}_x(a) - \mathrm{ev}_{x_0}(a) } < \varepsilon \right\},
    \]
    where $x_0\in X$, $a\in\mathcal{A}$ and $\varepsilon>0$ form a subbasis for the topology of pointwise convergence on $\mathrm{Im}(\Phi)$. The preimage by $\Phi$ of such a set is 
    \[
    \left\{ x \in X \colon \abs{a(x) - a(x_0)} < \varepsilon \right\}
    \]
    which is Borel since $a$ is a Borel map. Therefore $\Phi$ is Borel and injective.

    In order to chek that $\Phi$ is measure-preserving, i.e.~that $\Phi_\ast \lambda = \eta$, observe that for any $a \in \mathcal{A}$ and $x \in X$ we have
    \[
    \rho(a) \circ \Phi (x) = \rho(a)(\mathrm{ev}_x) = \mathrm{ev}_x(a) = a(x).
    \]
    It follows that for $a \in \mathcal{I}$ and $f = \rho(a)$ we have
    \[
        \int_X a d \lambda = \int_X \rho(a) \circ \Phi \, d \lambda = \int_Y \rho(a) d\Phi_\ast \lambda = \int_Y f d \Phi_\ast \lambda.
    \]
    On the other hand, if we further assume $f \in C_c(Y)$, Equation \eqref{eq: equality of integrals for the measure} yields the following:
    \[
    \int_X a d \lambda = \int_Y \rho(a) d \eta  = \int_Y f d \eta.
    \]
    We conclude that for all $f \in C_c(Y)$, $\int_Y f d \Phi_\ast \lambda = \int_Y f d \eta$, which by uniqueness 
    in Theorem \ref{thm: Riesz} yields \(\Phi_*\lambda=\eta\) as wanted.

    We now define the $G$-action $\widetilde{\alpha}$ on $\mathcal{L}^{\infty}(X)$, which is the precomposition by the inverse: $\widetilde{\alpha}(g, a)(x) = a(\alpha(g\inv,x))$. 
    This allow us to define the desired $G$-action $\beta$ on on $Y=\spec(\mathcal A)$ by 
    \[
    \beta(g,y)(a) \coloneqq y(\widetilde{\alpha}(g\inv,a)).
    \]
    We easily check that it is a left action on $Y$:
    for any character $y\in Y$ and any two elements $g$ and $h$ in $G$ we have
    \begin{align*}
    \beta(g,\beta(h,y))(a) & = \beta(h,y) ( \widetilde{\alpha}(g\inv, a) )\\
    & = y ( \widetilde{\alpha}(h\inv, \widetilde{\alpha}(g\inv, a)) )\\
    & = y ( \widetilde{\alpha}((gh)\inv,a) ) \\
    & = \beta(gh,y)(a).
    \end{align*}
    Moreover, for $y = \mathrm{ev}_x$ we have $\beta(g,y)(a) = \mathrm{ev}_x(\widetilde{\alpha}(g\inv ,a)) = a(\alpha(g,x)) = \mathrm{ev}_{\alpha(g,x)}(a)$, that is to say that, $\Phi$ is $G$-equivariant: $\Phi(\alpha(g,x)) = \beta(g,\Phi(x))$.

    In order to conclude, we have to check that $\beta$ is continuous and preserves $\eta$. The fact that it is measure-preserving is a direct consequence of the uniqueness in Theorem \ref{thm: Riesz}. Indeed, $\mathcal{I}$ is $G$-invariant and for any function $f$ in $\mathcal{I}$,
    \begin{align*}
    \int_X \widetilde{\alpha}(g,f)(x)d\lambda(x)  = \int_X f( \alpha(g\inv,x))d\lambda(x)
     = \int_X f(x) d\lambda(x).
    \end{align*}
    This means that $G$ preserves the integral of elements of $\mathcal{I}$, and by (\ref{eq: equality of integrals for the measure}) this implies that it preserves the integral of elements of $C_c(Y)$ (with regards to $\eta$). By uniqueness, the pushforward of $\eta$ by the action of any element of $G$ is equal to $\eta$. The $G$-action $\beta$ is therefore measure-preserving.

    Finally, the continuity of $\beta$ is obtained in the exact same way as in Step 4 of the proof of Theorem \ref{thm: chara existence spatial model}. 
\end{proof}
\begin{rem}
It is tempting to try to add the following third condition to the above theorem:
\begin{itemize}
\item[\emph{(iii)}] \emph{There is a continuous $G$-action on a Polish space $Y$ and a Borel injective map $\Phi: X\to Y$
        such that $\Phi_*\lambda$ is locally finite and $\Phi(\alpha(g,x))=\beta(g,\Phi(x))$ for all $x\in X$ and all $g\in G$.}
\end{itemize}
Clearly \textit{(i)} implies \textit{(iii)}, but we don't know if the converse is true, although we suspect it is not.
As we will see in section \ref{sec: non radonable actions}, the only examples of actions not satisfying \textit{(i)} that we have actually fail \textit{(iii)}.
\end{rem}

\subsection{The case of locally compact groups}{\label{sec: spatial Radon model lcsc}}

As in the previous section, we use the space $\mathcal L^\infty(X)$ of all bounded Borel functions $X\to \C$,
endowed with the norm of uniform convergence $\norm{f}_\infty=\sup_{x\in X} \abs{f(x)}$. 
We also use convolution as defined in Section \ref{sec: convolution}, noting that 
the proof of Lemma \ref{lem:convoGcont} yields the following statement for everywhere defined functions.
\begin{lem}
\label{lem:convoGcont curly L}
Consider a locally compact Polish group $G$, and a spatial measure-preserving action of $G$ on a standard $\sigma$-finite space $(X,\lambda)$.
Let $\delta : G \rightarrow \mathbb{R}_+$ be continuous of integral $1$ and compactly supported. Then, for any $f\in \mathcal L^{\infty}(X) \cap \mathcal L^{1}(X,\lambda)$, the function $\delta * f$ is in $\mathcal L^{\infty}(X) \cap \mathcal L^{1}(X,\lambda)$ and is $G$-continuous.\qed
\end{lem}

We can now prove Theorem \ref{thmi: locally compact spatical into continuous Radon} after first recalling its statement.

\begin{thm}\label{thm: lc polish embed into radon}
    Let $G$ be a locally compact Polish group, and denote by $\alpha$ a spatial measure-preserving $G$-action on a standard $\sigma$-finite space $(X,\lambda)$. Then there exists a locally compact Polish space $Y$ endowed with a continuous action $\beta$ and a Borel injection $\theta : X \to Y$ such that the pushforward measure $\theta_*\lambda$ is Radon and for all $x \in X$ and all $g \in G$,
    \[
    \theta(\alpha(g,x)) = \beta(g,\theta(x)).
    \]
\end{thm}

\begin{proof}
    By  Item \eqref{item:universalcompactGspace} of Theorem \ref{thm: bk spatial model} we may as well assume that $X$ is compact and $\alpha$ is continuous (this is also a consequence of an earlier and easier result of Varadarajan, see \cite[Thm.~3.2]{varadarajanGroupsAutomorphismsBorel1963}). We also denote by $d$ a compatible metric on $X$
    and fix a sequence $(x_n)$ enumerating a dense subset of $X$.
    By Theorem \ref{thm: admits model is equivalent to functions separate points}, it is sufficient to prove that there exists a countable set $\mathcal D$ of $G$-continuous functions in $\mathcal L^\infty(X)\cap \mathcal L^1(X,\lambda)$ that separates points. We will construct this set $\mathcal D$  using convolution with regard to the Haar measure $m$ on $G$.\\

    We begin by choosing for every $\varepsilon>0$ an open neighborhood $\mathcal{N}_\varepsilon$ of $e_G$ in $G$ such that for all $g\in\mathcal{N}_\varepsilon$ and all $x\in X$ we have $d(\alpha(g\inv,x),x) < \varepsilon$. Such a neighborhood exists by continuity of $\alpha$ and compactness of $X$. We then fix a continuous function
    $\delta_\varepsilon : G \rightarrow \R_+$
    of integral $1$ with a compact support included in $\mathcal{N}_\varepsilon$.

    Let us fix $n \in \N$ and $\varepsilon > 0$, our first aim is to define countably many integrable functions separating the points of the open ball $B(x_n,\varepsilon)$ from the points in $X \setminus B(x_n,3\varepsilon)$. 
    By the choice of $\delta_\varepsilon$, for any $x \in B(x_n,\varepsilon)$ and any $y \in X \setminus B(x_n,3\varepsilon)$, we have
    \begin{equation}{\label{eq: convolution separating dense sequence of points}}
    \left\{
    \begin{array}{l}
        \displaystyle{\delta_\varepsilon * \mathds{1}_{B(x_n,2\varepsilon)}(x) 
        = \int_G \delta_\varepsilon(g)\mathds{1}_{B(x_n,2\varepsilon)}(\alpha(g\inv,x))dm(g) 
        = \int_G \delta_\varepsilon(g) dm(g) = 1 }\vspace{0.1cm}\\
        \displaystyle{\delta_\varepsilon * \mathds{1}_{B(x_n,2\varepsilon)}(y) =
        \int_G \delta_\varepsilon(g)\mathds{1}_{B(x_n,2\epsilon)}(\alpha(g\inv,y))dm(g) 
        = 0}
    \end{array}
    \right.
    \end{equation}
    However the above function $\delta_\varepsilon * \mathds{1}_{B(x_n,2\varepsilon)}$ could very well fail to be integrable. We therefore set $X = \sqcup_{k \in \N} X_k$, with $\lambda(X_k) = 1$ for any $k$, and define
    \[
    A_{n,k,\varepsilon} = X_k \cap B(x_n,2\varepsilon),
    \]
    so that $\mathds{1}_{A_{n,k,\varepsilon}} \in \mathcal{L}^{1}(X,\lambda)$, for any positive integer $k$, and hence $\delta_\varepsilon*\mathds{1}_{A_{n,k,\varepsilon}} \in \mathcal{L}^{1}(X,\lambda)$ for all $\varepsilon>0$ by Lemma \ref{lem:convoGcont curly L}.

    In order to obtain a countable set of functions, we enumerate $\Q_{>0}$ as $\Q_{>0}=\{\varepsilon_i\colon i\in\N\}$. We can finally define our countable set of  functions:
    \[
    \mathcal{D} \coloneqq  \left\{ \delta_{\varepsilon_i} * \left( \sum_{k=0}^N \mathds{1}_{A_{n,k,\varepsilon_i}} \right) \colon i,n,N \in \N \right\}.
    \]
    By Lemma \ref{lem:convoGcont curly L} we have $\mathcal{D} \subseteq \mathcal{L}^{\infty}(X) \cap \mathcal{L}^{1}(X,\lambda)$ and $\mathcal D$ only contains $G$-continuous functions.\\

    We now prove that $\mathcal{D}$ separates the points of $X$. To this end, let $x$ and $y$ be in $X$ such that for all $f \in \mathcal{D}$ we have $f(x) = f(y)$.
    Fix $i$ and $n$ in $\N$. We have $B(x_n,2\varepsilon_i) = \bigsqcup_{k \in \N} A_{n,k,\varepsilon_i}$, hence we have the pointwise convergence 
    \[
    \sum_{k=0}^N \mathds{1}_{A_{n,k,\varepsilon_i}} \underset{N \rightarrow + \infty}{\longrightarrow} \mathds{1}_{B(x_n,2\varepsilon_i)}.
    \]
    Therefore by dominated convergence (using $\delta_{\varepsilon_i}$ as a dominating function) we have:
    \begin{align*}
        \delta_{\varepsilon_i} * \left( \sum_{k = 0}^{N} \mathds{1}_{A_{n,k,\varepsilon_i}} \right)(x)  &= \int_G \delta_{\varepsilon_i}(g) \left( \sum_{k = 0}^{N} \mathds{1}_{A_{n,k,\varepsilon_i}} \right)(\alpha(g\inv,x))dm(g) \\
         &\underset{N \to \infty}{\longrightarrow} \int_G \delta_{\varepsilon_i}(g) \mathds{1}_{B(x_n,2\varepsilon_i)}(\alpha(g\inv,x))dm(g)
         = \delta_{\varepsilon_i} * \mathds{1}_{B(x_n,2\varepsilon_i)}(x). 
    \end{align*}
    The same holds for $y$, and thus we have
    \[
    \delta_{\varepsilon_i} * \mathds{1}_{B(x_n,2\varepsilon_i)}(x) = \delta_{\varepsilon_i} * \mathds{1}_{B(x_n,2\varepsilon_i)}(y).
    \]
    By \eqref{eq: convolution separating dense sequence of points}, the function $\delta_{\varepsilon_i} * \mathds{1}_{B(x_n,2\varepsilon_i)}$ takes the value $1$ on $B(x_n,\varepsilon_i)$ and the value $0$ on $X \setminus B(x_n,3\varepsilon_i)$, so the previous equality ensures that for any $i,n \in \N$, we can never have simultaneously $d(x_n,x)<\varepsilon_i$ and $d(x_n,y)>3\varepsilon_i$. By density of $(x_n)$ and the fact that $(\varepsilon_i)$ takes arbitrarily small values, we  have $x=y$, which concludes the proof.
\end{proof}

\subsection{Local finiteness on \texorpdfstring{$X$}{X} itself}\label{sec: loc compact Radon vs loca fin real}

We now present a nice consequence of Theorem \ref{thm: lc polish embed into radon}
which was pointed out to us by Nachi Avraham-Re'em, allowing us to answer Question \ref{qu: loc fin Polish real}
positively when the acting group $G$ is locally compact Polish.

\begin{cor}\label{cor: loc fin model for lc groups}
	Let $G$ be a locally compact Polish group and $\alpha:G\times X\to X$ be a 
	spatial measure-preserving $G$-action on a standard $\sigma$-finite space $(X,\lambda)$.
	Then there is a Polish topology $\tau_X$ on $X$
    inducing its standard Borel structure
    such that $\alpha$ is a continuous action with respect to $\tau_X$
	and $\lambda$ is locally finite. 
\end{cor}
\begin{proof}
	Identifying $X$ to its image $\Phi(X)$ via the map $\Phi$ provided by Theorem \ref{thm: lc polish embed into radon},
	we can assume that $X$ is an $\alpha$-invariant Borel subset of a locally compact Polish space $(Y,\tau_Y)$ 
	on which $\lambda$ extends to a Radon measure $\eta$,
	that $\alpha$ extends to a continuous $\eta$-preserving action $\beta$,
    and that the standard Borel structure on $X$ is induced by the Borel $\sigma$-algebra of $\tau_Y$. 
	By \cite[Thm.~5.1.5]{beckerDescriptiveSetTheory1996}, there is a Polish topology $\tau'_Y$ on $Y$ which 
	refines  $\tau_Y$ such that $X$ is $\tau'_Y$-open, $\beta$ is $\tau'_Y$-continuous
    and $\tau'_Y$ generates the same Borel $\sigma$-algebra as $\tau_Y$. 
    
    Let us denote by $\tau_X$ the Polish topology induced by $\tau'_Y$ on $X$ and check that $\tau_X$ is as desired.
    First, since $X$ is $\tau'_Y$-open and $\tau'_Y$ is Polish, we have that $\tau_X$ is Polish as well by \cite[Thm.~3.11]{kechrisClassicalDescriptiveSet1995}.
    Moreover $\tau_Y$ and $\tau'_Y$ both induce the Borel $\sigma$-algebra of $Y$, and since $\tau_Y$
    induces the standard Borel space structure of $X$,
    so does $\tau_X$.
    Since $\beta$ is $\tau'_Y$-continuous and $\alpha$ is the restriction of $\beta$ to $X$, 
    $\alpha$ is $\tau_X$-continuous.
	Finally since $\eta$ is locally finite, its restriction $\lambda$ to $X$ is locally finite for the topology induced by 
    $\tau_Y$
	on $X$, so $\lambda$ is locally finite for the finer Polish topology $\tau_X$ induced by $\tau'_Y$  as well.
\end{proof}
In the above proof, we had to give up the local compactness of the ambient space $Y$, 
and it it natural to aks whether this can be circumvented.
We now observe that even for countable discrete groups, local compactness of the standard Borel 
space onto which they act may be impossible to obtain.

\begin{prop}\label{prop: no loc compact real}
	There is a countable discrete group $\Gamma$ and a Borel measure-preserving $\Gamma$-action on a standard $\sigma$-finite 
	space $(X,\lambda)$ such that $X$ cannot be endowed with a locally compact Polish topology so that the 
	action becomes continuous. 
\end{prop}
\begin{proof}
	Consider, for every pair of rationals $(q,r)$ with $r>0$, the Borel involutive bijection $I_{q,r}$ of the standard Borel space $\R\setminus \Q$ given by: for all $x\in\R\setminus\Q$, 
	\[
	I_{q,r}(x)=\left\{\begin{array}{cl}
		x+r & \text{if }x\in (q-r,q)\\
		x-r & \text{if }x\in(q,q+r)\\
		x &\text{otherwise}.
	\end{array}\right.
	\]
	Let $\Gamma$ be the group generated by all involutions $I_{q,r}$, which is naturally acting on $X=\R\setminus\Q$ in 
	a Borel manner, we claim this is the action we seek. First, this action preserves the measure $\lambda$ induced by 
	the Lebesgue measure on $\R\setminus\Q$ since it has the same orbits as the measure-preserving 
    $\Q$-action by translation on $\R\setminus\Q$.
	
	Now suppose by contradiction that $\tau$ is a locally compact Polish topology such that the $\Gamma$-action on 
	$X$ is continuous. 
	For every pair $(q,r)$ of rationals with $r>0$, notice that the support of $I_{q,r}$ (which is by definition the set of points not fixed by $I_{q,r}$) is equal to 
	$(q-r,q+r)\cap \R\setminus \Q$, and has to be $\tau$-open by continuity of the $\Gamma$-action. 
	This shows that $\tau$ refines the topology induced by $\R$ on $\R\setminus\Q$.
	Since $\tau$ is locally compact Polish, we can write $X=\bigcup_n K_n$ where each 
	$K_n$ is compact, and since $\tau$ refines the Polish topology of $\R\setminus\Q$, we deduce 
	that $\R\setminus\Q$ is $\sigma$-compact, a contradiction since every compact subset thereof has empty interior.
\end{proof}

\begin{rem}
	The key property of the action of $\Gamma$ in the above proof is that the supports generate some 
	non locally compact Polish topology. 
	Actions whose supports generate a Hausdorff topology have to be totally non free, meaning
	that the map taking a point to its stabilizer subgroup is injective. As observed by Vershik (see \cite[Sec.~2.3]{vershikTotallyNonfreeActions2012}), any two totally non-free actions are abstractly isomorphic if and only if their stabilizer maps have the same range. In particular, if we remove from the action 
	of the theorem an additional orbit $\Q+x$, we obtain for every $x\in\R$ a Borel action $\alpha_x$ on $\R\setminus(\Q\cup \Q+x)$, and $\alpha_x$ is  abstractly 
	isomorphic to $\alpha_y$ if and only if they are equal (which is of course equivalent to $x\in\Q+y$).

	This argument cannot work in the setup of free actions : all the orbits of free actions are the same,
	and a Hilbert hotel argument 
	shows that every free Borel action of a countable group is Borel isomorphic to any of its restrictions to the
	the complement of countably many orbits.
    We refer the reader to \cite[Sec.~3.3]{frischRealizationsCountableBorel2023} for some positive results on 
    Borel free actions of countable groups.
\end{rem}

\subsection{A counterexample when \texorpdfstring{$G=\mathfrak S_\infty$}{G}}\label{sec: non radonable actions}

Let us take $G = \mathfrak{S}_\infty$ for the whole section. 
Our aim is to show that Theorem \ref{thm intro: spatial conjugacy to nice action} is false for such a group, in particular Theorem \ref{thm: lc polish embed into radon} that we just proved
does not hold for such a group as well. 

The action $\alpha$ that we consider is the $\mathfrak S_\infty$-action on $X = \{ 0,1 \}^{\mathbb{N}}$ by permutation of the coordinates. Let us also fix a sequence $(p_n)$ of weights, with $p_n \in \left] 0 , 1 \right[$ for all $n$ in $\N$. We also require that $p_n \neq p_m$ for $n \neq m$, and that $p_n \rightarrow \frac{1}{2}$. Consider the measures $\mu_n$ defined by :
\[
\mu_n \coloneqq (p_n \delta_1 + (1-p_n) \delta_0)^{\otimes \N},
\]
and define the measure $\lambda$ on $X$ by
\[
\lambda \coloneqq \sum_{n \in \N} \mu_n .
\]

Let $X_\infty$ be the subspace of $X$ consisting of sequences that take infinitely many times the value $0$, and infinitely many times the value $1$. First note that $\Sinf$ acts transitively on $X_\infty$. We have $\lambda(X\setminus X_\infty) = 0$, as for any $n$ in $\N$ we have $\mu_n(X \setminus X_\infty)  =0$, and thus, up to a null set, we can restrict to $X_\infty$.

Each $\mu_n$ is a probability measure on $X$, so $\lambda$ is an infinite measure. We now verify that $\lambda$ is $\sigma$-finite.
Indeed, the strong law of large numbers ensures us that if we define a family $(X_n)$ of Borel sets by
\[
X_n := \left\{     x \in X\colon \dfrac{ \abs{ \left\{  k \in\{0,\dots,m-1\}\colon x_k =1 \right\}} }{m} \underset{m \rightarrow + \infty}{\longrightarrow} p_n   \right\},
\]
that is to say that $X_n$ is the set of the sequences of $X$ with a proportion $p_n$ of $1$, then for all $n$ in $\N$ we have 
$\mu_n(X_n) = 1$.
We thus have
\[
\displaystyle{\lambda\left(X \setminus  \bigsqcup_{n \in \mathbb{N}} X_n    \right) = 0 }.
\]
The condition $p_n \neq p_m$, for $n \neq m$, ensures us that the $X_n$ are disjoint sets, and therefore $\lambda$ is $\sigma$-finite.

The following lemma provides an obstruction to having a continuous locally finite model.

\begin{lem}{\label{lem:infinite measure neighborhoods}}
    Let $\mathcal{N}$ be a neighborhood of the identity in $\Sinf$, let $x\in X_\infty$. Then we have $\lambda\left(\alpha(\mathcal{N},x) \right) = + \infty$.
\end{lem}
	
\begin{proof}
For every $n$ let $H_n$ denote the open subgroup of $\mathfrak S_\infty$ given by 
\[
H_n=\{\sigma\in\Sinf\colon \forall i\in \{0,\cdots n\},\, \sigma(i)=i\}.
\]
Then the sequence $(H_n)$ is a neighborhood basis of the identity in $\Sinf$, so we can find some $H_N$ contained in our neighborhood $\mathcal{N}$. For any $x_0$ in $X_\infty$, we have 

\[
\alpha(H_N,x)=
\left\lbrace  y \in X_\infty \colon \forall i \in \{ 0 , \ldots , N \},\, y_i = x_i   \right\rbrace
\]
so since $p_n\to \frac 12$ we have \[
\mu_n(\alpha(H_N,x))   \underset{n \rightarrow + \infty}{\longrightarrow} \left(\frac{1}{2} \right)^{N+1}.
\]
Therefore,
\[
\lambda (\alpha(H_N,x)) = \sum_{n \in \mathbb{N}}\mu_n (\alpha(H_N,x)) = + \infty,
\]
which concludes the proof.
\end{proof}

\begin{rem}
    By continuity of the action, for any $x$ in $X_\infty$, for any neighborhood $\mathcal{V}$ of $x$ there exists $N$ in $\N$ such that $\alpha(H_N,x) \subseteq \mathcal{V}$. So by Lemma \ref{lem:infinite measure neighborhoods} and the density of $X_\infty$, all nonempty open subsets
    of $X$ have infinite measure.
\end{rem}

We have all the tools to show that Theorem \ref{thm intro: spatial conjugacy to nice action} can fail badly for non locally compact Polish groups. 

\begin{prop}{\label{prop: Sinfty action non Radonable}}
    The spatial $\Sinf$-action $\alpha$ on $(X,\lambda)$ defined above cannot be spatially isomorphic to a continuous action on a Polish space $Y$ endowed with a locally finite measure $\eta$.
\end{prop}

\begin{proof}
    Suppose that we have a continuous action $\beta$ on $(Y,\eta)$ with $\eta$ locally finite and $Y$ Polish, and that $\Phi: X_0\to Y$ is a spatial isomorphism between $\alpha$ and $\beta$. 
    Let us take some  $x\in X_0$, then there is an open set $U$ containing $y$ such that $\eta(U) < \infty$. 
    By continuity of $\beta$, there exists a neighborhood of the identity $\mathcal{N}\subseteq G$  satisfying $\beta(\mathcal{N},\Phi(x)) \subseteq U$. However by Lemma \ref{lem:infinite measure neighborhoods}, 
    $\lambda(\alpha(\mathcal N, x))=+\infty$. 
    In particular, $\lambda(\alpha(\mathcal N,x)\cap X_0)=+\infty$. 
    By the equivariance condition satisfied by $\Phi$, the set $\beta(\mathcal N,\Phi(x))$ contains the infinite measure set $\Phi(\alpha(\mathcal N,x)\cap X_0)$, so it has infinite measure, contradicting the fact that it is contained in the finite measure set $U$.
\end{proof}

We now define a second $\Sinf$-action $\beta$ as the action on $\bigsqcup_n(\{0,1\}^\N,\mu_n)$ by permutation of the coordinates, but this time in different distinct copies of $\{0,1\}^\N$. This will yield an interesting infinite measure-preserving example of non spatial boolean isomorphism (see Section \ref{sec: two realizations} for Tsankov's example in the finite measure case). 

\begin{prop}\label{prop: inf mp pas spati iso}
     The two $\Sinf$-actions $\alpha$ on $(\{ 0,1 \}^{\mathbb{N}},\lambda)$ and $\beta$ on $\bigsqcup_n(\{0,1\}^\N,\mu_n)$ are booleanly isomorphic, but not spatially isomorphic.
\end{prop}

\begin{proof}
    As $X_n$ denotes the $\mu_n$-conull set of sequences with a proportion $p_n$ of coordinates equal to $1$, it is possible to send $X_n$ to a copy of itself in the $n$-th copy of $\{ 0,1 \}^{\mathbb{N}}$ by defining
    \[
    \Phi: \bigsqcup_{n \in \N}{X_n}\subseteq\{0,1\}^\N\longrightarrow\bigsqcup_{n \in \N}{X_n}\subseteq\bigsqcup_{n \in \N}\{0,1\}^\N.
    \]
    For a fixed permutation $\sigma$, $\lambda$-almost any sequence $x$ in $X_\sigma = \bigsqcup_n{X_n}$ satisfies $x \in X_k \Rightarrow \alpha(\sigma,x) \in X_k$, and therefore $\Phi(\alpha(\sigma,x)) = \beta(\sigma,\Phi(x)) \in \Phi(X_k)$.
    The actions $\alpha$ and $\beta$ are not spatially isomorphic however, as $\alpha$ only has one orbit in $\{ 0,1 \}^{\mathbb{N}}$, but $\beta$ cannot send an element of $\Phi(X_n)$ to $\Phi(X_m)$ whenever $n \neq m$.
\end{proof}

\begin{rem} More generally, by Proposition \ref{prop: Sinfty action non Radonable}, $\alpha$ cannot admit a continuous Radon model, so any boolean isomorphism between $\alpha$ and another continuous action on a space endowed with a Radon measure cannot be a spatial isomorphism.
\end{rem}

\begin{rem}
    Similar ideas work for instance when $G=\Aut(\Q,<)$, identifying $\N$ to $\Q$ and replacing $X_\infty$ by the set of all sequences $(x_q)\in \{0,1\}^\Q$ such that the set $\{q\in \Q\colon x_q=1\}$ is both dense and codense in $\Q$ (meaning that for all rationals $q_1<q_2$ there are $q$ and $q'$ in the interval $(q_1,q_2)$ such that $x_q=1$ and $x_{q'}=0$). Indeed, this new set has $\mu_n$ measure $1$, and a back-and-forth argument shows that it consists of a single $\Aut(\Q,<)$-orbit. It would be nice to identify more precisely the class of non-archimedean groups for which there exists a spatial infinite measure-preserving action which cannot be spatially isomorphic to a continuous action on a locally finite Polish measured space. 
\end{rem}

The example we built in Proposition~\ref{prop: inf mp pas spati iso} is clearly non ergodic, hence the following 
natural question.

\begin{question}
    Does $\Sinf$ admit two \emph{ergodic} infinite measure-preserving actions which are booleanly isomorphic, but not spatially isomorphic?
\end{question}

\section{Poisson point processes and Lévy groups}

\subsection{Construction of the Poisson point process}

Let $X$ be a Polish space, we can endow its space of closed subsets $\mathcal F(X)$ with the  lower Vietoris topology (also called lower semi-finite topology in \cite{michaelTopologiesSpacesSubsets1951}), which is the second-countable $T_0$ topology generated by declaring, for every $U\subseteq X$ open,
that the set 
\[
\mathcal V_U=\{F\in\mathcal F(X)\colon F\cap U\neq\emptyset\}
\]
is open. The associated Borel $\sigma$-algebra is called the Effros $\sigma$-algebra, and it turns $\mathcal F(X)$ into a standard Borel space, see for instance \cite[Cor.~9.3]{chenNotesQuasiPolishSpaces2021}. 

\begin{lem}\label{lem: counting is Borel}
    Let $U\subseteq X$.
    The function which takes $F\in\mathcal F(X)$ to $\abs{U\cap F}$ is lower semi-continuous if we endow $\mathcal F(X)$ with the lower Vietoris topology, in particular it is Borel. 
\end{lem}
\begin{proof}
Suppose $\abs{U\cap F_0}\geq k$, then we find $U_1$,...,$U_k$ disjoint opens subsets of $U$ such that $F_0\cap U_i$ is non empty, and this will be true of any $F\in \bigcap_{i=1}^k\mathcal V_{U_i}$, hence the result.
\end{proof}
\begin{rem}In particular, the set of all $F\in\mathcal F(X)$ such that $F\cap U$ is finite is Borel. 
\end{rem}

Let us now recall that the Poisson law of intensity $0 < \lambda < \infty$ is the probability measure $\Poisson\lambda$ on $\N$ given by
$$
\Poisson\lambda(\{k\})=e^{-\lambda}\frac{\lambda^k}{k!}
$$
and that its expected value is equal to $\lambda$. We also define the degenerate Poisson law of intensity $0$ as the Dirac measure at $0$.

We now define on any Polish space endowed with an atomless locally finite Borel measure $\lambda$ \emph{the} \textbf{Poisson point process} of intensity $\lambda$  as the probability measure whose existence and uniqueness are granted by the following elementary version of  \cite[Thm.~3]{avraham-reemPoissonianActionsPolish2024}. 

\begin{thm}\label{thm: PPP}
    Let $\lambda$ be a locally finite atomless Borel measure on a Polish space $X$. There is a unique probability measure $\PPP X \lambda$ on $\mathcal F(X)$ such that for all $U\subseteq X$ open of finite measure, if the random variable $F$ has law $\PPP X \lambda$, then the random variable
    $\abs{F\cap U}$ follows the Poisson law of parameter $\lambda(U)$. 
\end{thm}

\begin{rem}
    The theorem does not yield what is usually called a Poisson point process per se (see \cite[Def.~1.1]{avraham-reemPoissonianActionsPolish2024}).
    In order to have one, we would need to further check that for all $A\subseteq X$ Borel,
    the map $F\in\mathcal F(X)\mapsto \abs{A\cap F}$ is $\PPP X \lambda$-Lebesgue measurable, 
    and that it follows a Poisson law of parameter $\lambda(A)$. 
    Since we don't need this stronger property, we do not prove it and refer the reader to the proof of
    \cite[Thm.~3]{avraham-reemPoissonianActionsPolish2024}.
\end{rem}

\begin{proof}
    In order to see the uniqueness, we first remark that closed subsets of the form 
    $\mathcal C_U=\{F\in\mathcal F(X)\colon F\cap U=\emptyset\}=\mathcal F(X)\setminus\mathcal V_U$
    form a $\pi$-system (because $\mathcal C_{U_1}\cap\mathcal C_{U_2}=\mathcal C_{U_1\cup U_2}$) which generates the $\sigma$-algebra. 
    We can then compute the probability of $\mathcal C_U$ as follows.
    \begin{itemize}
        \item If $\lambda(U)$ is finite, then by construction  $\mathcal C_U$ 
        has probability $e^{- \lambda(U)}$.
        \item Otherwise, by local finiteness we may write $U =\bigcup_n U_n$ 
        where $U_n\subseteq U_{n+1}$ and each $U_n$ has finite measure. 
        We thus have $\lim_n \lambda(U_n)=\lambda(U)=+\infty$, moreover
        $\mathcal C_U$ is contained in each $\mathcal C_{U_n}$, 
        which has probability $e^{- \lambda(U_n)}\to 0$. We conclude that 
        $\mathcal C_U$ has probability $0$. 
    \end{itemize}
    By the monotone class theorem, this proves the uniqueness of the probability measure. \\

    For the existence, let us first prove it when $\lambda$ is finite.
    In this case, we may as well restrict ourselves to defining the measure on the subset $\mathcal P_f(X)$ of finite subsets of $X$, which is Borel by Lemma \ref{lem: counting is Borel}.

    The map $\Phi:\N\times X^\N\to \mathcal P_f(X)$ which takes $(k,(x_n)_{n\geq 0})$ to $\{x_n\colon n<k\}$ is easily seen to be continuous if we put on $\mathcal P_f(X)$ the topology induced by the lower Vietoris topology on $\mathcal F(X)$, in particular it is Borel. 
    Renormalize $\lambda$ to a probability measure $\widetilde \lambda= \frac\lambda{\lambda(X)}$.
    Endow $\N\times X^\N$ with the probability measure $\mu=\Poisson{\lambda(X)} \otimes \widetilde \lambda^{\otimes\N}$. We claim that the pushforward measure $\Phi_*\mu$ is the desired Poisson Point Process.

    In order to see this, first note that since $\lambda$ is atomless, $\mu$-almost all $(k,(x_n))$ satisfy that $x_n\neq x_m$ whenever $n\neq m$. 
    It follows that given $U\subseteq X$ open and $l\in\N$, if for $k\geqslant l$ we denote by $\mathcal P_l(\{0,\dots,k-1\})$ the set  of subsets of $\{0,\dots,k-1\}$ of cardinality $l$, the event $\abs{F\cap U}=l$ has probability
    $$
    \Phi_*\mu(\abs{F\cap U}=l) =  e^{-\lambda(X)}\sum_{k\geq l}\left (\frac{\lambda(X)^k}{k!}\sum_{K\in \mathcal P_l(\{0,\dots,k-1\})}\left(\frac{\lambda(U)}{\lambda(X)}\right)^l\left(\frac{\lambda(X\setminus U)}{\lambda(X)}\right)^{k-l}\right)
    $$
    Letting $x=\lambda(X)$ and $u=\lambda(U)$, we rewrite this as 
    \begin{align*}
    e^{- x} \sum_{k\geqslant l} \frac{1}{k!} {\binom{k}{l}} u^l (x-u)^{k-l} & = e^{- x} \sum_{k\geqslant l} \frac{ u^l(x-u)^{k-l}}{l!(k-l)!} \\
    & = e^{- x} \frac{ u^l}{l!} e^{x-u}\\
    & = e^{- u} \frac{ u^l}{l!}. 
    \end{align*}
    Since $u=\lambda(U)$, our probability measure $\Phi_*\mu$ has the required distribution and hence the existence is proven when $\lambda$ is a finite measure.\\

    Let us now deal with the case where $\lambda$ is infinite. By local finiteness write $X=\bigcup_n V_n$ where each $V_n$ is open of positive finite measure and satisfies $V_n\subseteq V_{n+1}$.
    Each $V_n$ is Polish for the induced topology, and if we denote by $\lambda_n$ the restriction of $\lambda$ to $V_n$, what we have done so far grants us a unique Borel probability measure $\PPP{V_n}  \lambda$ on $\mathcal P_f(V_n)$ such that for all $U\subseteq V_n$ open, if $F$ has law $\PPP{V_n}\lambda$ then the random variable 
    $\abs{F\cap U}$ has law $\Poisson{\lambda(U)}$.

    For $n\in\N$, consider the continuous hence Borel map $\pi_n :$ $\mathcal P_f(V_{n+1})\to \mathcal P_f(V_n)$
    which maps $F$ to $F\cap V_n$. Form the projective limit
    \[
    Y=\varprojlim \mathcal P_f(V_n)=\left\{(F_n)\in \prod_{n \in \N}\mathcal P_f(V_n)\colon \forall n\in \N, \, F_n=F_{n+1}\cap V_n\right\}, 
    \]
    which is a Borel subset of the standard Borel space $\prod_n \mathcal P_f(V_n)$ by \cite[Thm.~V.2.5]{ParthasarathyProbabilityMeasuresonMetricSpaces1967}.

    For every $n\in\N$, it is a straightforward consequence of uniqueness that $\pi_{n*}\PPP{V_{n+1}}\lambda=\PPP{V_n}\lambda$, so by Kolmogorov's consistency theorem (see e.g.~\cite[Thm.~V.3.2]{ParthasarathyProbabilityMeasuresonMetricSpaces1967}) we have a unique Borel probability measure $\mu$ on $Y$ such that for all $n\in\N$, $p_{n*}\mu=\PPP{V_n}\lambda$, where $p_n:Y\to \mathcal P_f(V_n)$ is the projection on coordinate $n$.

    Let us finally consider the map $\Psi: Y\to \mathcal F(X)$ which takes $(F_n)$ to the increasing union $\bigcup_n F_n$. First note that $\Psi$ is well defined since $\bigcup_n F_n$ has finite intersection in each $V_n$, and the $V_n$'s exhaust $X$ so $\bigcup_n F_n$ is a discrete subset of $X$ without accumulation points, hence closed (and countable).
    Next, $\Psi$ is a Borel (actually continuous) map since $\bigcup_n F_n \cap U\neq \emptyset$ if and only if there exists $n$ such that $F_n\cap (V_n\cap U)\neq \emptyset$.
    We claim that $\PPP{X}\lambda \coloneqq \Psi_*\mu$ is the desired Poisson Point Process.
    Indeed, if $U$ is an arbitrary open subset of $X$ of finite measure, we have that
    $\abs{F\cap U}=k$ if and only if there exists $N$ such that for all $n\geq N$, $\abs{F\cap (U\cap V_n)}=k$.
    For a fixed $N$, the event  that for all $n\geq N$, $\abs{F\cap (U\cap V_n)}=k$ is the intersection of a decreasing sequence of events which
    has probability
    $$ \mathbb{P}_{\lambda}^X \left( \forall n \geqslant N,\, \abs{F \cap (U \cap V_n)} = k \right) = \lim_{n\to+\infty} e^{- \lambda(U\cap V_n)}\frac{ \lambda(U\cap V_n)^k}{k!}=e^{-\lambda(U)}\frac{\lambda(U)^k}{k!}.$$
    It follows that the event $\abs{F\cap U}=k$ also has probability $e^{-\lambda(U)}\frac{\lambda(U)^k}{k!}$ as wanted.
\end{proof}

\begin{rem}
The proof of uniqueness is actually a version of Rényi’s theorem (see for instance \cite[Thm.~6.10]{lastLecturesPoissonProcess2017}), namely it shows that $\PPP X \lambda$ is  uniquely defined by 
the fact that for every finite measure open subset $U\subseteq X$, we have 
\[
    \PPP X \lambda(F\cap U=\emptyset)=e^{- \lambda(U)}.
\]
\end{rem}

\begin{rem}Since the measure is locally finite, $\PPP X  \lambda$-almot every  $F\in\mathcal{F}(X)$ is locally finite, that is to say that every point of $F$ admits a neighborhood whose intersection with $F$ is a singleton.
\end{rem}

\subsection{Application to actions of Lévy groups}

\begin{defi}
    A sequence $(X_n,d_n,m_n)$ of metric spaces with probability measure is a \textbf{Lévy family} if it satisfies the following condition: if $(A_n)$ is a sequence of measurable subsets of $(X_n)$ such that $\liminf m_n(A_n) > 0$, then for any $\varepsilon > 0$, $\lim_n m_n( (A_n)_{\varepsilon} ) = 1$, where $(A_n)_{\varepsilon}\coloneqq \{x\in X_n\colon d_n(x,A_n)<\varepsilon\}$.
    
    A Polish group $G$ is a \textbf{Lévy group} if there is a sequence of compact subgroups $(K_n)$ with $K_n \subseteq K_{n+1}$, such that $\bigcup_n K_n$ is dense in $G$, and such that the $(K_n,m_n)$, where $m_n$ is the normalized Haar measure, is a Lévy family when equipped with a right-invariant compatible metric $d$ on $G$ (the choice of such a $d$ does not matter).
\end{defi}

We refer the reader to \cite{pestovDynamicsInfinitedimensionalGroups2006} for more about Lévy groups.
Glasner, Tsirelson and Weiss proved the following fundamental result.

\begin{thm}[{\cite[Thm. 1.1]{glasnerAutomorphismGroupGaussian2005}}]{\label{LevyProba}}
    Every measure-preserving spatial action of a Lévy group $G$ on a possibly atomic standard probability space $(X,\mu)$ is trivial, \textit{i.e.} the set of fixed points $\{ x \in X \colon \forall g \in G ,  g \cdot x =  x \}$ is conull.
\end{thm}

We obtain an analogous result for continuous actions of Lévy groups on Polish spaces with a infinite atomless $\sigma$-finite measure $\lambda$, under the assumption that the measure is locally finite.
This result was first obtained by Avraham-Re'em and Roy, 
see \cite[Thm.~5]{avraham-reemPoissonianActionsPolish2024}, and was stated in our introduction as Theorem \ref{thmi: Lévy infinite}.

\begin{thm}{\label{thm: Lévy infinite}}
    Every continuous measure-preserving spatial action of a Lévy group $G$ on a Polish space $X$ endowed with a $\sigma$-finite atomless locally finite measure $\lambda$ is trivial, \textit{i.e.} the set of fixed points $\{ x \in X \mid \forall g \in G , g \cdot x =  x \}$ is conull.
\end{thm}
\begin{proof}
    Suppose by contradiction that we have an action as above on $(X,\lambda)$, but that the action is not trivial. Consider the standard Borel space $\mathcal{F}(X)$ endowed with the Poisson Point Process on $(X,\lambda)$ of intensity $\lambda$, that we denoted by $\PPP X\lambda$. Then $(\mathcal F(X), \PPP X\lambda)$
    is a standard probability space, and we now show that $G$ acts spatially on it. 
    
    Since every element of $G$ defines a homeomorphism of $X$, it sends any closed subset of $X$ to a closed subset of $X$, and in particular $G$ acts on $\mathcal{F}(X)$.
    Since $g\cdot \mathcal V_U=\mathcal V_{g\inv U}$ and
    $G$ acts by homeomorphism on $X$, we have that $G$ acts by homeomorphisms on $\mathcal F(X)$ for the lower Vietoris topology.
    
    Let us show that this action is  continuous for the lower Vietoris topology. Given $U\subseteq X$ open, we need to show that the  set of couples $(g,F)$ such that $g\cdot F\cap U\neq \emptyset$ is open.
    So take $(g_0,F_0)$ such that $g_0\cdot F_0\cap U\neq\emptyset$, then $F_0\cap g_0\inv\cdot U\neq \emptyset$, then let $x_0\in F_0\cap g_0\inv \cdot U$.
    In particular $g_0\cdot x\in U$, so
    by continuity of the action of $G$ on $X$, there is a neighborhood $V$ of $g_0$ and $W$ of $x_0$ such that $V\cdot W\subseteq U$. Consider the neighborhood $V\times \mathcal V_{W}$ of $(g_0,F_0)$, then any element $(g,F)$ in the neighborhood $V\times \mathcal V_{W}$ of $(g_0,F_0)$ satisfies $F\cap W\neq \emptyset$, and hence $g\cdot F\cap g\cdot W\neq \emptyset$. Since $g\in v$ we have $g\cdot W\subseteq U$, so we conclude that $gF\cap U\neq \emptyset$ as wanted. In particular, the $G$-action on $\mathcal F(X)$ is Borel.
    
    This action preserves $\PPP X \lambda$ by uniqueness. Indeed consider the random variable $F$ of law $\PPP X\lambda$, for any finite measure open subset $U$ of $X$ and any $g$ in $G$ we have 
    \begin{align*}
    \PP_\lambda^X( \abs{F \cap U} = k) & = e^{- \lambda(U)}\frac{ \lambda(U)^k}{k!}\\
    & = e^{- \lambda(g\inv \cdot U)}\frac{ \lambda(g\inv \cdot U)^k}{k!}\\
    & = \PP_\lambda^X(\abs{F \cap g\inv \cdot U} = k)\\
    & = \PP_\lambda^X(\abs{g \cdot F \cap U} = k).
    \end{align*} 
    Being a spatial measure-preserving action on a standard probability space, the $G$-action on $\mathcal F(X)$ is trivial by Theorem \ref{LevyProba}. We will now see why this cannot be.
    
    We first claim that we can find an open subset $U$ of $X$ with $0 < \lambda(U) < \infty $ and $g\in G$ which satisfies $ U \cap g \cdot U = \emptyset$.
    Indeed, since $G$ acts continuously non-trivially on $(X,\lambda)$ the set of non fixed points $X_0\coloneqq\{ x \in X \mid \exists  g \in G , g \cdot x \neq  x \}$ has positive measure.  
    By continuity of the $G$-action on $X$, for every $x\in X_0$ we find $g_x\in G$ and an open subset $U_x\subseteq X$ such that $ U_x\cap g_x \cdot U_x=\emptyset$.
    By Lindelöf's theorem, the cover $(U_x)_{x\in X_0}$ admits a countable subcover $(U_{x_i})_{i\in\N}$. Since $\lambda(X_0)>0$, we conclude that some $U_{x_i}$ satisfies $\lambda(U_{x_i})>0$, and thus taking $g= g_{x_i}$ and $U=U_{x_i}$, we have $\lambda(U)>0$ and $U\cap g \cdot  U=\emptyset$ as desired.
    
    Now remember that if $\mathcal C_{U}=\{F\in\mathcal F(X)\colon F\cap U=\emptyset\}$, we have $\PPP X \lambda (\mathcal{C}_U) = e^{-\lambda(U)}<1$. We then have
    \[
    \PPP X \lambda (\mathcal{C}_U \cap g \cdot \mathcal{C}_U) = \PPP X \lambda (\mathcal{C}_U \cap \mathcal{C}_{g\inv \cdot U}) = \PPP X \lambda (\mathcal{C}_{U \sqcup g\inv \cdot U}) = e^{-\lambda(U \sqcup g\inv \cdot U)} = \PPP X \lambda (\mathcal{C}_U)^2,
    \]
    so $\PPP X \lambda (\mathcal{C}_U \cap g \cdot \mathcal{C}_U)<\PPP X \lambda(\mathcal{C}_U)$,
    yielding the non-triviality of the action and thus
    the desired contradiction.
\end{proof}

\begin{rem}
The hypothesis that $\lambda$ has no atoms is not important: 
if it does, then $G$ acts by permutation on these atoms, but the
associated Bernoulli shift must be trivial by the Glasner-Tsirelson-Weiss theorem,
so $G$ acts trivially on the atomic part of $\lambda$.
\end{rem}

\section{The natural action of \texorpdfstring{$\Aut(X,\lambda)$}{Aut(X,lambda)} admits no spatial model}
\label{sec: whirly}

Whirly actions were introduced in the probability measure-preserving context by Glasner, Tsirelson and Weiss, who showed such actions cannot have spatial models.
In this last section we record the natural extension of this definition to the infinite measure-preserving setup, and show that whirly actions admit no spatial models. As an example, we prove that the natural boolean action of $\Aut(X,\lambda)$ is whirly and hence does not admit a spatial model.

\begin{defi}[{\cite[Sec.~3]{glasnerAutomorphismGroupGaussian2005}}]
	A boolean measure-preserving action $\alpha$ of a Polish group $G$ on $(X,\lambda)$ is \textbf{whirly} if for all sets $A$ and $B$ of positive measure, for every neighborhood $\mathcal{N}$ of $e_G$, there exists a element $g$ in $\mathcal{N}$ such that $\lambda(A \cap \alpha(g,B)) > 0$.    
\end{defi}

\begin{prop}{\label{prop: Aut whirly}}
	Let $G = \Aut(X,\lambda)$. The natural boolean action of $G$ on $(X,\lambda)$ is whirly.
\end{prop}

\begin{proof}
	Let $A$ and $B$ be two Borel subsets of $X$ of positive measure, and fix $\varepsilon > 0$. We consider $A_0 \subseteq A$ and $B_0 \subseteq B$ two disjoint Borel subsets, with $\lambda(A_0) = \lambda(B_0) < \varepsilon/2$. We define a element $S$ of $\Aut(X,\lambda)$ as follows:
	\[
	\left\{
	\begin{array}{l}
		S_{\restriction X \setminus (A_0 \cup B_0)} = \mbox{id}_{X \setminus (A_0 \cup B_0)}  \\
		S(A_0) = B_0 \\
		S(B_0) = A_0,
	\end{array}
	\right.
	\]
	and see that for $\varepsilon$ small, $S$ is close to the identity on $X$, for the weak topology. This concludes the proof, as $\lambda(A \cap SB) > 0$.
\end{proof}

Let us now explain the link between whirly actions and $G$-continuity. We have the following infinite measure version of \cite[Prop. 3.3]{glasnerAutomorphismGroupGaussian2005}. The proof is actually the same, but we provide it for the reader's convenience.

\begin{prop}{\label{prop: whirly no spatial model}}
	Consider a boolean measure preserving action $\alpha$ of a Polish group $G$ on $(X,\lambda)$. If the action is whirly, then all $G$-continuous functions are constant. Moreover, such an action cannot admit a spatial model.
\end{prop}

\begin{proof}
	We first explain why the existence of $G$-continuous non-constant functions taking values in $\R$ contradicts the fact that the action is whirly. .
	
	Let $f$ be a $G$-continuous non-constant function taking values in $\R$. There exists $a<b$ in $\R$ such that $A := f\inv(]-\infty,a[)$ and $B := f\inv(]b,+\infty[)$ have positive measure. By $G$-continuity, there exists a neighborhood $\mathcal{N}$ of $e_G$ such that any $g$ in $\mathcal{N}$ satisfies $\norm{f\circ\alpha(g_n\inv, \cdot ) - f}_{\infty} < b-a$. This implies that $\lambda(A \cap \alpha(g,B)) = 0$, contradicting the fact that the action is whirly.
	
	Now for a complex valued function, one simply has to consider the real and imaginary parts, as they are $G$-continuous whenever $f$ is. Assuming that $f$ is non-constant implies that at least one or the other is non-constant, thus implying that the action cannot be whirly.

	For the second part of the statement, suppose the action has a spatial model on $(Y,\eta)$. By The Becker-Kechris theorem (Theorem \ref{thm: bk spatial model}), it is possible to assume that the $G$-action on $Y$ is continuous, with $Y$ compact Polish. The support of $\eta$ is closed, thus compact, call it $K$. Since $G$ acts continuously on $Y$, any function in $C_c(K)$ is $G$-continuous by Lemma \ref{lem: G continuous from continuous action on lc}, and thus any non-constant function in $C_c(K)$ gives a non-constant $G$-continuous function, a contradiction. 
\end{proof}

Combining Propositions \ref{prop: Aut whirly} and \ref{prop: whirly no spatial model}, we obtain the following.

\begin{prop}{\label{prop: Aut no spatial model}}
	Let $(X,\lambda)$ be a standard $\sigma$-finite space. The natural boolean action of $\Aut(X,\lambda)$ on $(X,\lambda)$ cannot admit a spatial model.
\end{prop}

\begin{rem}\label{rem: spat mod for Aut}
    Theorem \ref{thm: Lévy infinite} applies to $\Aut(X,\lambda)$ when $\lambda$ is locally finite. Indeed, Giordano and Pestov have proved in \cite[Thm. 4.2]{giordanoSomeExtremelyAmenableGroups2007} that it is a Lévy group when endowed with its weak topology. However, Proposition \ref{prop: Aut no spatial model} does not follow from the above result, as there are no topological requirements on the spatial model.
\end{rem}

\begin{rem}
	The previous discussion holds in the non-singular setup as well, considering the tautological boolean non-singular  action on $(X,\lambda)$ of the group  $\Aut(X,[\lambda])$ of all \textit{non-singular} bijections (identified up to equality on a conull set). The action is whirly in the exact same sense as above, and the proof of Proposition \ref{prop: whirly no spatial model} adapts in a straightforward manner to non-singular whirly actions.
\end{rem}

\bibliographystyle{alphaurl}
\bibliography{BibliSigma} 
\listoffixmes

\end{document}